\documentclass[12pt]{amsart}

\usepackage{hyperref}
\usepackage{color}

\usepackage{graphicx}
\usepackage{xypic}

\usepackage{amsmath,amsfonts,amssymb}
\usepackage{colortbl}

\usepackage{pstricks}
\usepackage{pst-plot}

\theoremstyle{plain}
\newtheorem*{theorem*}{Theorem}
\newtheorem*{lemma*} {Lemma}
\newtheorem*{corollary*} {Corollary}
\newtheorem*{proposition*} {Proposition}
\newtheorem*{conjecture*}{Conjecture}

\newtheorem{theorem}{Theorem}[section]
\newtheorem{lemma}[theorem]{Lemma}

\newtheorem{proposition}[theorem]{Proposition}
\newtheorem{conjecture}[theorem]{Conjecture}

\theoremstyle{definition}

\theoremstyle{remark}
\newtheorem*{remark}{Remark}
\newtheorem*{remarks}{Remarks}

\newtheorem*{claim}{Claim}

\theoremstyle{definition}
\newtheorem{defn}[theorem]{Definition}

\newtheorem{question}{Question}

\newrgbcolor{deepgreen}{0.1 0.5 0.1} 

\def\Q{\Bbb{Q}}
\def\F{\Bbb{F}}
\def\Z{\Bbb{Z}}
\def\R{\Bbb{R}}
\def\C{\Bbb{C}}
\def\N{\Bbb{N}}
\def\LL{\mathcal{L}}
\def\be{\begin{equation}} \def\ee{\end{equation}}
\def\id{\operatorname{id}}

\def\eps{\epsilon}
\def\sign{\operatorname{sign}}
\def\null{\operatorname{null}}

\def\l{\lambda}

\def\MM{\mathcal{M}}
\def\ZZ{\mathcal{Z}}

\def\ext{\operatorname{Ext}}
\def\sm{\setminus}

\def\bp{\begin{pmatrix}}
\def\ep{\end{pmatrix}}
\def\bn{\begin{enumerate}}
\def\en{\end{enumerate}}

\def\ba{\begin{array}}
\def\ea{\end{array}}

\def\L{\Lambda}

\def\tpm{[t^{\pm 1}]}

\def\s{\sigma}

\def\ev{\operatorname{ev}}
\def\PD{\operatorname{PD}}

\def\lk{\operatorname{lk}}
\def\fr12{\frac{1}{2}}

\def\ol{\overline}

\def\S{\Sigma}
\def\BB{\mathcal{B}}

\def\CC{\mathcal{C}}

\def\PP{\mathcal{P}}

\def\realt{\R[t^{\pm 1}]}

\def\zt{\Z[t^{\pm 1}]}

\def\hom{\operatorname{Hom}}

\def\O{\Omega}

\def\vt{V^t}

\numberwithin{equation}{section}

\def\to{\mathchoice{\longrightarrow}{\rightarrow}{\rightarrow}{\rightarrow}}
\makeatletter
\newcommand{\shortxra}[2][]{\ext@arrow 0359\rightarrowfill@{#1}{#2}}
\def\longrightarrowfill@{\arrowfill@\relbar\relbar\longrightarrow}
\newcommand{\longxra}[2][]{\ext@arrow 0359\longrightarrowfill@{#1}{#2}}
\renewcommand{\xrightarrow}[2][]{\mathchoice{\longxra[#1]{#2}}%
  {\shortxra[#1]{#2}}{\shortxra[#1]{#2}}{\shortxra[#1]{#2}}}
\makeatother

\begin{document}

\title{The unknotting number and classical invariants I}

\date{\today}

\author{Maciej Borodzik}
\address{Institute of Mathematics, University of Warsaw, Warsaw, Poland}
\email{mcboro@mimuw.edu.pl}

\author{Stefan Friedl}
\address{Mathematisches Institut\\ Universit\"at zu K\"oln\\   Germany}
\email{sfriedl@gmail.com}

\thanks{The first author is supported by Polish MNiSzW Grant No N N201 397937 and also by the Foundation for Polish Science FNP}

\def\subjclassname{\textup{2010} Mathematics Subject Classification}
\expandafter\let\csname subjclassname@1991\endcsname=\subjclassname
\expandafter\let\csname subjclassname@2000\endcsname=\subjclassname


\begin{abstract}
Given a knot $K$ we introduce a new invariant coming from the Blanchfield pairing and we
show that it gives a  lower bound on the unknotting number of $K$.
This lower bound subsumes the  lower bounds given by the Levine-Tristram signatures, by the Nakanishi index
and it also subsumes the Lickorish obstruction to the unknotting number being equal to one.
Our approach in particular allows us to show for 25 knots with up to 12 crossings that their unknotting number is at least three,
most of which are very difficult to treat otherwise.
\end{abstract}

\maketitle

\section{Introduction}

Let $K\subset S^3$ be a knot. Throughout this paper a knot is always assumed to be oriented.
A \emph{crossing change} is one of the two local moves on a knot diagram given in Figure \ref{fig:cc}.
\begin{figure}[h]
\begin{center}
\begin{pspicture}(-7,-0.5)(6,2.2)
\psline(-4.5,0)(-3.75,0.75)\psline[arrowsize=6pt]{->}(-3.25,1.25)(-2.5,2)
\psline[arrowsize=6pt]{->}(-4.5,2)(-2.5,0)
\rput(-3.5,0){\psscalebox{1.25}{``$+$''}}
\psline(3,2)(3.75,1.25)\psline[arrowsize=6pt]{->}(4.25,0.75)(5,0)
\psline[arrowsize=6pt]{->}(3,0)(5,2)
\rput(4,0){\psscalebox{1.25}{``$-$''}}
\psline[linestyle=dashed,arrowsize=3pt]{->}(-2,1.1)(2,1.1)\rput(0,1.35){\psscalebox{0.8}{negative crossing change}}
\psline[linestyle=dashed,arrowsize=3pt]{<-}(-2,0.9)(2,0.9)\rput(0,0.65){\psscalebox{0.8}{positive crossing change}}
\end{pspicture}
\end{center}
\label{fig:cc}
\caption{Negative and positive crossing change.}
\end{figure}
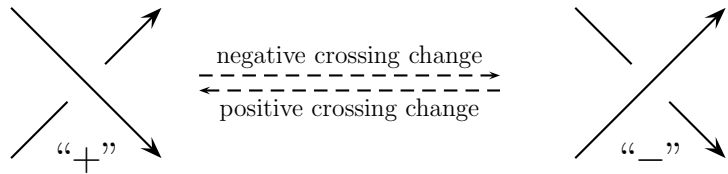
The \emph{unknotting number $u(K)$} of a knot $K$ is defined to be the minimal number of crossing changes
necessary to turn $K$ into the unknot.
The unknotting number  is one of the most elementary invariants of a knot, but also one of the most intractable.
Whereas upper bounds can be found readily using diagrams, it is much harder to find non--trivial lower bounds.

In the paper we will for the most part study a closely related invariant, namely the \emph{algebraic unknotting number $u_a(K)$},
which is defined to be the minimal number of crossing changes necessary to turn $K$ into a knot with trivial Alexander polynomial.
By \cite{Fo93} and \cite{Sae99} this is equivalent to the original definition by Murakami \cite{Muk90}
given in terms of `algebraic unknotting moves' on Seifert matrices.
It is clear that $u(K) \geq u_a(K)$, and that in general this is not an equality. For example for any non--trivial knot $K$ with trivial Alexander polynomial
we have $u(K)\geq 1$ and $u_a(K)=0$.

\subsection{Review of classical invariants}

Let us first fix some terminology. Let $F$ be
a Seifert surface for $K$ and let $v_1,\dots,v_n$ be
a collection of embedded simple closed curves  on
$F$ which represent a basis for $H_1(F;\Z)$. The corresponding Seifert matrix $V$ is defined as the matrix with $(i,j)$-entry given by
\[ \lk(v_i,v_j^+),\]
where   we denote by $v_j^+$ the positive push-off of $v_j$.
The S--equivalence class of the Seifert matrix is well-known to be an invariant of $K$
(see e.g.  \cite[Theorem~8.4]{Lic97} for details). The S--equivalence class of the Seifert matrix will be denoted by $V=V_K$.
By abuse of notation we will often denote by $V=V_K$ a representative of the $S$--equivalence class.
In this paper, by a \emph{classical invariant} of a knot we mean an invariant which is determined by $V_K$.

Given
a knot $K$ we denote by  $X(K)=S^3\sm \nu K$ the exterior of $K$ and we denote by $\Sigma(K)$ its branched cover.
We now give several well--known examples of classical invariants which will play a r\^ole
in the paper (in the following the matrix $V$ is a $2n\times 2n$ Seifert matrix for $K$):
\bn
\item The Alexander polynomial is defined as
 \[ \Delta_K(t)=t^{-n}\cdot \det(Vt-\vt)\in \zt. \]
 Note that $\Delta_K(t)$ is well-defined with no indeterminacy, and $\Delta_K(1)=1$.
\item The knot determinant $\det(K)=(-1)^n\det(V+V^t)$, which in this paper is viewed as a signed invariant.
\item  The isometry type of the linking pairing
\[ l(K)\colon H_1(\S(K);\Z)\times H_1(\S(K);\Z)\to \Q/\Z, \]
which is
isometric to the pairing
\[ \ba{rcl} \Z^{2n}/(V+V^t)\Z^n \times \Z^{2n}/(V+V^t)\Z^{2n}&\to& \Q/\Z \\
 (v,w)&\mapsto & v^t(V+V^t)^{-1}w.\ea \]
 We refer to \cite{Go78} for details.
 \item The Blanchfield pairing
\[ \lambda(K)\colon H_1(X(K);\zt)\times H_1(X(K);\zt)\to \Q(t)/\zt\]
which is a hermitian non--singular pairing on the Alexander module $H_1(X(K);\zt)$.
We refer to  Section \ref{section:seifertbf} for the definition.
 \item The \emph{Nakanishi index $m(K)$}, i.e. the
  minimal number of generators of the Alexander module $H_1(X(K);\zt)$.
\item Given $z\in S^1$ the \emph{Levine-Tristram signature} is defined as
\[  \s_z(K)=\sign(V(1-z)+\vt(1-z^{-1})).\]
(Note that $\s_{-1}(K)$ is just the ordinary knot signature $\s(K)$.)
\item Given $z\in \C\sm\{0,1\}$ the \emph{nullity} is defined as
\[ \eta_z(K)=\null(V(1-z)+\vt(1-z^{-1})),\]
furthermore $\eta_1(K)$ is by convention defined to be $0$.
\en

\subsection{Classical bounds for the unknotting number}
We will now quickly summarize all previous classical lower bounds on the unknotting number which are known to the authors.

The first lower bounds on the unknotting number go back to Wendt \cite{We37}, they are subsumed by the following inequality due to Nakanishi \cite{Na81}:
 \[  u(K) \geq  m(K). \]
 We discuss it in Section~\ref{section:nakanishi}

It has been known since the work of Murasugi \cite{Mus65} that  Levine--Tristram signatures give rise to lower bounds on the unknotting number.
In particular the following inequality holds:
 \[ u_a(K) \geq \mu(K):=\frac{1}{2}\left(\max\{ \eta_z(K)+\s_z(K) \, |\, z\in S^1\}+\max\{ \eta_z(K)-\s_z(K) \, |\, z\in S^1\}\right).\]
 This inequality is in all likelihood known to the experts, but we are not aware of a reference and we thus give a proof (together with a more refined statement) in Section \ref{section:lteta}.

By Saeki \cite[Proposition~4.1]{Sae99} the topological 4--ball genus $g_4^{top}(K)$ is a lower bound on the algebraic unknotting number $u_a(K)$.
Livingston \cite{Liv11} introduced a classical  invariant $\rho(K)$ which gives a lower bound on $g_4^{top}(K)$. In Section  \ref{section:livingston} we will slightly modify
Livingston's invariant to define a new classical invariant $\rho_{\zt}(K)$ which satisfies
\[
 g_4^{top}(K)\geq \rho_{\zt}(K)\geq \rho(K)\]
 and we will show that $n(K)\geq \rho_{\zt}(K)$.

We now recall several classical obstructions to a knot $K$ having `small' algebraic unknotting number.
 If $K$ can be unknotted using a single $\eps$--crossing change (with $\eps \in \{-1,1\})$), then by the work of Lickorish \cite{Lic85}  there exists a generator $h$ of $ H_1(\Sigma(K);\Z)$ such that
\[ l(h,h)=\frac{-2\eps }{\det(K)}\in \Q/\Z.\]
Recently Jabuka \cite{Ja09} also introduced an obstruction to the unknotting number being one, in Section \ref{section:jabuka} we will see that it is subsumed by the Lickorish obstruction.
Also note that  the Lickorish obstruction was generalized  by Fogel,  Murakami and Rickard (see \cite{Fo93,Muk90,Lic11}) in terms of the Blanchfield pairing.
Finally note that if $|\s(K)|=4$, then Stoimenow \cite[Proposition~5.2]{St04} gives a classical obstruction to $u_a(K)=2$ in terms of the determinant of $K$.
To the best of our knowledge the above is a complete list of lower bounds on the unknotting number given by classical invariants.

\begin{remarks}
\bn
\item
Lower
bounds on the unknotting number have also been obtained using
 gauge theory \cite{CoL86,KMr93}, Khovanov homology \cite{Ras10}
and Heegaard-Floer homology \cite{Ras03,OS03b,OS05,Ow08,Gr09,Sar10} and various other methods \cite{KM86,Kob89,ST89,Mi98,Tra99,St04,MQ06,GL06}.
Note though, that with the exception of the Rasmussen $s$--invariant,
the Ozsv\'ath--Szabo $\tau$--invariant and the Owens obstruction
most of the above are in fact obstructions to the unknotting number being equal to
one or two.
\item Without doubt, the most important result on unknotting numbers has been the resolution of the Milnor conjecture by Kronheimer and Mrowka \cite{KMr93}: the unknotting number of the $(p,q)$--torus knot equals $\frac{(p-1)(q-1)}{2}$. We also refer to \cite{OS03b,Ras10,Sar10} for alternative proofs.
Finally we refer to \cite{BW84} for an interesting pre--gauge theory discussion of the problem.
\en
\end{remarks}

\subsection{Definition of the invariant $n(K)$}
Given a hermitian $n\times n$-matrix $A$ over $\zt$ with $\det(A)\ne 0$  we denote by
$\lambda(A)$ the pairing
\[ \ba{rcl} \l(A)\colon \zt^n/A\zt^n \times \zt^n/A\zt^n&\to & \Q(t)/\zt \\
(a,b) &\mapsto & \ol{a}^t A^{-1} b,\ea \]
where we view $a,b$ as represented by column vectors in $\zt^n$. Note that $\lambda(A)$ is a non-singular, hermitian
pairing.

Let $K$ be a knot.
We  define $n(K)$ to be the minimal size of a hermitian matrix $A$ over $\zt$ such that
\begin{itemize}
\item $\lambda(A)\cong \lambda(K)$, i.e. $\lambda(A)$ is isometric to the Blanchfield pairing of $K$, and
\item the matrix $A(1)$ is congruent over $\Z$ to a diagonal matrix which has $\pm 1$'s on the diagonal.
\end{itemize}
In Section \ref{section:seifertbf} we will see that the Blanchfield pairing of  $K$ can indeed be represented
by such a matrix $A$, i.e. we will show that $n(K)$ is actually defined.
We will furthermore show that $n(K)\leq \deg \Delta_K(t)+1$.
Note that $n(K)=0$ if and only if the Alexander polynomial of $K$ is trivial.
Finally note that $n(K)$ is a classical invariant since the Blanchfield pairing is a classical invariant.

\subsection{The main theorem}

Our main theorem is the following.

\begin{theorem}\label{thm1}
Let $K$ be a knot which can be turned into an Alexander polynomial one
knot using $u_+$ positive crossing changes and $u_-$ negative crossing changes.
Then there exists a hermitian matrix $A(t)$ of size $u_++u_-$ over $\zt$ with the following two properties:
\bn
\item $\l(A(t))\cong \l(K)$,
\item $A(1)$ is a diagonal matrix such that $u_+$ diagonal entries are equal to $-1$ and $u_-$ diagonal entries are equal to $1$.
\en
In particular $u_a(K)\geq n(K)$.
\end{theorem}

In Section~\ref{section:classicallower} we will show that the lower bound on the algebraic unknotting number from Theorem~\ref{thm1} contains, to the best knowledge of the authors, all previous classical lower bounds to the unknotting number.
This result can be summarized in the following theorem:

\begin{theorem}\label{thm:subsume}
The invariant $n(K)$ subsumes the following  unknotting obstructions:
\bn
\item the Nakanishi index (see Section~\ref{section:nakanishi}),
\item the invariant $\mu(K)$ (see Section~\ref{section:lteta}),
\item $\rho_{\zt}(K)$ and in particular Livingston's invariant $\rho(K)$ (see Section~\ref{section:livingston}),
\item the Fogel--Murakami--Rickard obstruction (see Section~\ref{section:fogelmurakami}),
\item the Lickorish obstruction and the Jabuka obstruction (see Section~\ref{section:lickorish} and Section~\ref{section:jabuka}),
\item the Stoimenow obstruction (see Section~\ref{section:stoimenow}).
\en
In particular all of the above give lower bounds on the \emph{algebraic} unknotting number.
\end{theorem}
The precise statements and the proofs are given in the indicated parts of Section~\ref{section:classicallower}.

\begin{remark}
The fact that $n(K)$ subsumes all other classical lower bounds does not invalidate those earlier bounds, since all but the first one are directly computable,
whereas at the moment there is no algorithm to calculate $n(K)$ in general.
\end{remark}

\begin{remark}
The fact that the earlier classical lower bounds on the unknotting number give in fact lower bounds on the algebraic unknotting number
can also at times be deduced from reading carefully the original proofs.
\end{remark}

Fogel \cite{Fo94} proved the following  remarkable partial converse to Theorem \ref{thm1}.

\begin{theorem}
If $n(K)=1$, then $u_a(K)=1$.
\end{theorem}

Fogel's proof is constructive in the sense that in many cases, given a knot $K$ with $n(K)=1$, one can actually find explicitly a diagram and
a crossing change which turns $K$ into an Alexander polynomial one knot.
We refer to \cite[Section~3]{Fo93} and \cite[Section~4]{Fo94} for more details. The results of Fogel make plausible the following conjecture.

\begin{conjecture}\label{conj:nua}
For any knot $K$ we have
\[ n(K)=u_a(K).\]
\end{conjecture}

We plan to investigate this conjecture in a future paper.

\begin{remark}
Theorem \ref{thm:subsume}  can in particular be viewed as evidence towards Conjecture \ref{conj:nua}.
\end{remark}

\subsection{Diagrammatic comparison of classical invariants}

In order to show how the newly defined invariant $n(K)$ fits into the bigger picture of knot invariants,  we present in
Figure \ref{fig:ddoki}
a diagram which shows the relationship between various topological and classical invariants.
Beyond the invariants introduced above we will also use the following topological invariants:
\[
\ba{rcl}
g_3(K)&=&\mbox{minimal genus of a  surface in $S^3$ cobounding the knot $K$},\\
g_4^{smooth}(K)&=&\mbox{minimal genus of a smooth surface in $D^4$ cobounding the knot $K$},\\
g_4^{top}(K)&=&\mbox{minimal genus of a locally flat surface in $D^4$ cobounding the knot $K$},\ea \]
and the following classical invariants
\[ \ba{rcl}
\eta(K)&:=&\max\{ \eta_z(K) \, |\, z\in\C\sm \{0\}\},\\
m_{\R}(K)&:=& \mbox{minimal number of generators of $H_1(X(K);\realt)$},\\
n_{\R}(K)&:=& \mbox{minimal size of a hermitian matrix over $\realt$ representing}\\
&& H_1(X(K);\R\tpm)\times H_1(X(K);\R\tpm)\to \R(t)/\R\tpm.
\ea
\]
It is straightforward to see  that $m_{\R}(K)=\eta(K)$.
In a future  paper \cite{BF12b} we will show  that furthermore
\[ n_\R(K)=\max\{\mu(K),\eta(K)\}.\]
In Figure \ref{fig:ddoki} we use the following notation:
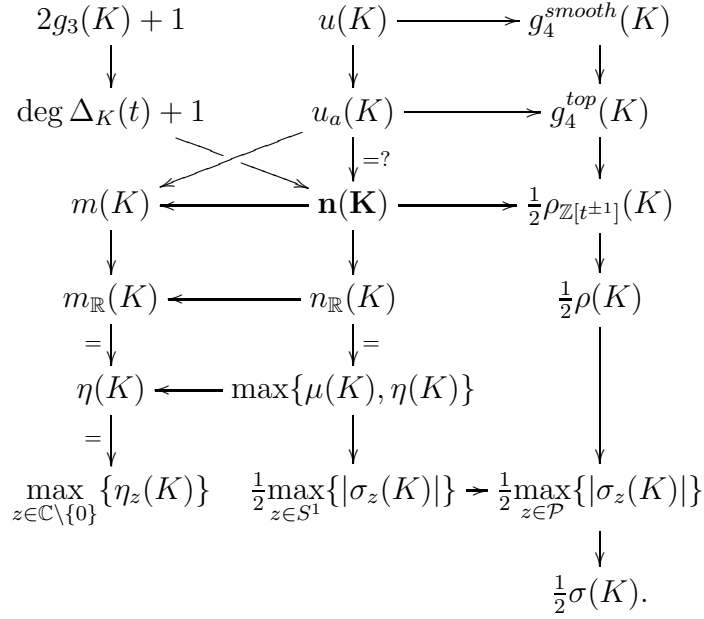
\begin{figure}
\[ \xymatrix @R=0.2in @C=0.01in{
2g_3(K)+1\ar[d] & u(K) \ar[d] \ar[r]& g_4^{smooth}(K)\ar[d] \\
\deg\Delta_K(t)+1\ar[dr]|\hole & u_a(K)\ar[dl] \ar[d]^{=?} \ar[r]& g_4^{top}(K)\ar[d] \\
m(K)\ar[d]& \ar[l] \mathbf{n(K)}\ar[d]\ar[r]& \frac{1}{2}{\rho}_{\zt}(K)\ar[d] \\
m_\R(K)\ar[d]_= &\ar[l] n_\R(K)\ar[d]^=&\frac{1}{2}\rho(K)\ar[dd] \\
\eta(K)\ar[d]_=& \max\{\mu(K),\eta(K)\}\ar[l] \ar[d] && \\
\underset{z\in \C\sm\{0\}}{\max}\{\eta_z(K)\}& \frac{1}{2}\underset{{z\in S^1}}{\max}\{|\s_z(K)|\}\ar[r]& \frac{1}{2}\underset{z\in \PP}{\max}\{|\s_z(K)|\}\ar[d] \\
&&  \frac{1}{2}\s(K).}\]
\caption{Diagrammatic summary  of known invariants.}\label{fig:ddoki}
\end{figure}
given two knot invariants $f$ and $g$ we  write $f(K) \to g(K)$ if $f(K)\geq g(K)$ for every knot $K$. Furthermore we denote by $\PP$ the set of all prime power roots of unity.

The existence of an arrow is in all cases either well-known, or a tautology or it follows
from the results in this section.
If two invariants are not related by a concatenation
of arrows, then in most cases it is known that they are unrelated.
If an arrow is not decorated by an `$=$' sign, then in most cases it is known that the invariants are indeed
not equal.

\subsection{Applications and examples}

Our understanding of the relation between the $n(K)$ invariant and the presentation matrix for the linking pairing of the double
branched cover (cf. Section~\ref{section:linkingform}, especially Lemma~\ref{lem:atlinkingform}) allows us to provide new computable
 obstructions for $u(K)=2$ and $u(K)=3$, which are related to Owens' obstruction from \cite{Ow08}.
 The idea behind the results in
Sections~\ref{section:obstwo} and \ref{section:obsthree} is the following. If the Blanchfield pairing can be realized by a $n\times n$
matrix over $\Z[t^{\pm}]$, then there exists an $n\times n$ integer matrix of a certain type which represents the linking pairing $l(K)$
of the double branched cover. Up to  congruence there exist finitely many such matrices, which furthermore
in many cases can be listed explicitly.
It is then straightforward to verify whether or not $l(K)$ can be represented by any of these matrices.

Among knots with up to $12$ crossings there are $25$ knots with  $m(K)\leq 2$ and $\mu(K)\leq 4$,
but where our approach shows that $n(K)\geq 3$. Out of these $25$ knots the Stoimenow obstruction detects four knots,
to the best of our knowledge no other classical obstruction applies to these 25 knots.
Also, in most cases the Rasmussen $s$--invariant and the  Ozsv\'ath--Szab\'o $\tau$--invariant can not detect the unknotting number.
We furthermore checked the $u(K)=3$ obstruction for all knots with
up to 14 crossings with $|\s(K)|=6$ and $m(K)\leq 3$. We found that it applies
precisely to two such knots, namely $14n_{12777}$ and $14a_{4637}$.
We have not yet implemented the obstruction to $u(K)=n$ for higher values of $n$.

Our new obstruction to $n(K)=2$ now allows us to completely determine the algebraic unknotting number for all knots with up to 11 crossings.
Details are given in Section \ref{section:ex} and in \cite{BF12a}.


\subsection*{Acknowledgment.}
This paper greatly benefitted from conversations with Baskar Balasubramanyam, Slaven Jabuka, Raymond Lickorish,
Brendan Owens, Andrew Ranicki,
Andr\'as Stipsicz and  Alexander Stoi\-me\-now.
We  would like to thank Micah Fogel for sending us his thesis and Hitoshi Murakami for supplying us with a copy of \cite{Muk90}.
The website `knotinfo' \cite{CL11} which is maintained by Jae Choon Cha and Chuck Livingston
has been an invaluable tool for finding examples and testing our algorithms.
We are also very grateful that Alexander Stoimenow provided us with braid descriptions for knots up to 14 crossings and
we wish to thank Julia Collins for help with obtaining Seifert matrices from the braid descriptions.

Finally we also would like to express our gratitude
to the Renyi Institute for its hospitality and to the London Mathematical Society for a travel grant.

\section{Proof of Theorem \ref{thm1}}\label{section:proofthm1}

Throughout Section \ref{section:proofthm1} we write
\[ \L:=\zt \mbox{ and } \Omega:=\Q(t).\]
As usual we also identify $\L$ with the group ring of $\Z$.

\subsection{Poincar\'e duality and the universal coefficient spectral sequence}\label{section:basics}

In this section we will collect several facts which we will use continuously throughout the paper.

Throughout the paper $X$ will always denote a manifold whose first homology group is isomorphic to $\Z$.
We denote the infinite cyclic covering of $X$ by $\pi\colon \widehat{X}\to X$.
Given a submanifold $Y\subset X$  we write $\widehat{Y}=\pi^{-1}(Y)$.
Note that $\Z$ is the deck transformation group of $\widehat{X}$. This defines a canonical left action
of $\L=\Z[\Z]$ on $C_*(\widehat{X},\widehat{Y};\Z)$. Given any $\L$-module $N$ we now define
\[ H^*(X,Y;N):=H_*(\hom_{\L}(C_*(\widehat{X},\widehat{Y};\Z),N))\]
and
\[ H_*(X,Y;N):=H_*(\ol{C_*(\widehat{X},\widehat{Y};\Z)}\otimes_{\L} N).\]
(Here, and throughout the paper, given a module $H$ over $\L$ we denote
by $\ol{H}$ the module with the involuted $\L$-structure, i.e. $\ol{H}=H$ as abelian groups,
but multiplication by $p\in \L$ in $\ol{H}$ is the same as multiplication by $\ol{p}$ in $H$.)
In particular we can consider the modules
$H_*(X,Y;\L)$, $H^*(X,Y;\L)$, $H_*(X,Y;\O)$, and $H^*(X,Y;\O)$.
When $Y=\emptyset$, then we will suppress $Y$ from the notation.

Note that the quotient field $\O$ is flat over the ring $\L$.
In particular we have
$H_*(X,Y;\O)\cong H_*(X,Y;\L)\otimes_{\L}\O$ and $H^*(X,Y;\O)\cong H^*(X,Y;\L)\otimes_{\L}\O$.

Suppose that $X$ is an $n$-manifold, then for any $\L$-module $N$ Poincar\'e duality
defines isomorphisms of $\L$-modules
\[ \ba{rcl} H_i(X,\partial X;N)&\cong & \ol{H^{n-i}(X;N)}\\
H_i(X;N)&\cong & \ol{H^{n-i}(X,\partial X;N)}.\ea \]

Finally we recall the  universal coefficient spectral sequence (UCSS),
we refer to  \cite[Theorem~2.3]{Lev77} for details. Let $N$ be any $\L$-module. Then the UCSS starts with
 $E_{p,q}^2=\ext^p_{\L}(H_q(X;\L),N)$ and converges
to $H^*(X;N)$. The differentials at the $r-$stage of this sequence have degree $(1-r,r)$.
Note that for any two $\L$-modules $H$ and $N$
the module $\ext^0_{\L}(H,N)$ is canonically isomorphic to
$\hom_{\L}(H,N)$. Also note that
\[ \ext^p_{\L}(H,N)=0\]
for any $p>2$ since $\L$ has cohomological dimension 2.
Finally  note that $\Z$, viewed as a $\zt$--module with trivial $t$--action, admits a resolution of length 1.
It now follows that $\ext^p_{\L}(\Z,N)=0$ for any $p>1$.

\subsection{Seifert matrices and Blanchfield pairings}\label{section:seifertbf}

Let $K\subset S^3$ be a knot. We  consider the following sequence of maps:
\be\label{equ:defbl}
\ba{rcl}
\Phi\colon H_1(X(K);\L)&\to& H_1(X(K),\partial X(K);\L)\\
&\to& \ol{H^2(X(K);\L)}\xleftarrow{\cong} \ol{H^1(X(K);\O/\L)}\\
&\to &\ol{\hom_{\L}(H_1(X(K);\L),\O/\L)}.\ea \ee
Here the first map is the inclusion induced map, the second map is Poincar\'e duality,
the third map comes from the long exact sequence in cohomology corresponding to the coefficients
$0\to \L\to \O\to \O/\L\to 0$, and the last map is the evaluation map.
It is well--known that the first map is an isomorphism, the second map is obviously an isomorphism,
and it follows from the UCSS (and the straightforward calculation that $\ext_{\L}^p(\Z,\O/\L)=0$ for $p\geq 1$) that the evaluation map is also an isomorphism.
It follows that the above maps thus  define a non-singular pairing
\[ \ba{rcl} \lambda(K)\colon  H_1(X(K);\L)\times H_1(X(K);\L)&\to& \O/\L\\
(a,b)&\mapsto & \Phi(a)(b),\ea \]
called the \emph{Blanchfield pairing of $K$}. This pairing is well-known to be hermitian, in particular   $\l(K)(a_1,a_2)=\ol{\l(K)(a_2,a_1)}$ and $\l(K)(\mu_1 a_1,\mu_2a_2)=\ol{\mu_1}\l(K)(a_1,a_2)\mu_2$
for $\mu_i\in \L, a_i\in H_1(X(K);\L)$.
We refer to  \cite{Bl57} for an alternative definition and for further details.

\begin{remark}
The Alexander polynomial $\Delta_K(t)$ is well-known to annihilate the Alexander module.
It now follows easily from the definition of $\l(K)$, that $\l(K)$ takes values in $\Delta_K(t)^{-1}\zt/\zt\subset \Q(t)/\zt$. This fact also follows from the description of the Blanchfield pairing in terms of Seifert matrices due to Kearton \cite[Section~8]{Ke75} which we recall below.
\end{remark}

Let  $V$ be any matrix of size $2k$ which is $S$--equivalent to a Seifert matrix for $K$.
Note that $V-\vt$ is antisymmetric and it satisfies $\det(V-\vt)=(-1)^k$.
It is well--known that, possibly after replacing $V$ by $PVP^t$ for an appropriate $P$, the following equality holds:
\be \label{equ:vvt} V-\vt= \bp 0 & \id_k \\ -\id_k &0 \ep.\ee
Following \cite[Section~4]{Ko89} we now define $A_K(t)$ to be the matrix
\begin{multline}\label{equ:akt}
\bp (1-t^{-1})^{-1}\id_k &0\\ 0&\id_k \ep V \bp \id_k&0\\0&(1-t)\id_k\ep +\\
+\bp \id_k &0 \\ 0&(1-t^{-1})\id_k \ep \vt\bp (1-t)^{-1}\id_k &0 \\ 0&\id_k
\ep.\end{multline}
Note that  the matrix $A_K(t)$  is a hermitian matrix defined over $\L$ and note that $\det(A_K(1))=(-1)^k$
(see \cite{Ko89}).
Also note that
\be \label{equ:aktlt} \bp (1-t^{-1})\id_k &0 \\ 0&\id_k\ep A_K(t)  \bp (1-t)\id_k &0 \\ 0&\id_k\ep= (1-t)V+(1-t^{-1})\vt.\ee
We now have the following proposition.

\begin{proposition}\label{prop:blakt}
Let $K$ be a knot and $A_K(t)$ as above, then
$\lambda(A_K(t))\cong  \lambda(K)$.
\end{proposition}

Note that the isometry type of the Blanchfield pairing in fact determines
the $S$-equivalence class of the Seifert matrix, see
\cite{Tro73} and \cite{Ran03}. In that sense the Blanchfield pairing is a `complete' classical invariant, i.e. it determines all other classical invariants.

\begin{proof}
First note that the Blanchfield pairing $\lambda(K)$ is isometric to the following pairing
(we refer to \cite[Section~8]{Ke75} for details):
\be \label{equ:ke75} \xymatrix{ \L^{2k}/(Vt-\vt)\L^{2k} \times \L^{2k}/(Vt-\vt)\L^{2k} \ar[rrr]^-{(t-1)(Vt-\vt)^{-1}}&&&  \O/\L.}\ee
The notation we use here, and similarly below, means that to $a,b\in \L^{2k}$
we associate $ \ol{a}^t (1-t)(Vt-\vt)^{-1}b$.
We now write $\L_0:=\Z[t,t^{-1},(1-t)^{-1}]$ and
we let
\[ P:= \bp t^{-1}\id_k &0\\ 0&(t-1)^{-1}\id_k \ep.\]
We consider the following commutative diagram:
\[ \xymatrix{
\L^{2k}/A_K(t)\L^{2k}\times \L^{2k}/A_K(t)\L^{2k}  \ar[d] \ar[rrr]^-{A_K(t)^{-1}} &&& \O/\L \ar[d]\\
\L_0^{2k}/A_K(t)\L_0^{2k} \times \L_0^{2k}/A_K(t)\L_0^{2k}\ar[d]^{(v,w)\mapsto (Pv,Pw)} \ar[rrr]^-{A_K(t)^{-1}} &&& \O/\L_0\ar[d] \\
\L_0^{2k}/PA_K(t)\L_0^{2k}\times \L_0^{2k}/PA_K(t)\L_0^{2k} \ar[d]^-{=}
\ar[rrr]^-{(PA_K(t)^{-1}\ol{P}^t)^{-1}} &&& \O/\L_0\ar[d] \\
\L_0^{2k}/(Vt-\vt)\L_0^{2k}\times \L_0^{2k}/(Vt-\vt)\L_0^{2k}  \ar[rrr]^-{(t-1)(Vt-\vt)^{-1}} &&& \O/\L_0 \\
\L^{2k}/(Vt-\vt)\L^{2k}\times \L^{2k}/(Vt-\vt)\L^{2k}  \ar[u]\ar[rrr]^-{(t-1)(Vt-\vt)^{-1}} &&& \O/\L.\ar[u]}\]
Here the top vertical maps and the bottom  vertical maps are induced by the inclusion $\L\to \L_0$.
Recall that multiplication by $t-1$ induces an isomorphism of $\L^{2k}/(Vt-\vt)\L^{2k}$
and of $\L^{2k}/A_K(t)\L^{2k}$ (see \cite{Lev77}). It follows that the two aforementioned maps are isomorphisms of $\L$-modules.
For the third vertical map we made use of the fact that
\[ PA_K(t)\ol{P}^t=(t-1)^{-1}(Vt-\vt)\]
and we used that
\[ P_AK(t)\L_0^{2k}=PA_K(t)\ol{P}^t\L_0^{2k}=(t-1)^{-1}(Vt-\vt)\L_0^{2k}=(Vt-\vt)\L_0^{2k}.\]
Since all vertical maps on the left in the above commutative diagram are isomorphisms we deduce from (\ref{equ:ke75}) that $\lambda(A_K(t))\cong \lambda(K)$.

\end{proof}

We conclude this section with the following lemma:

\begin{lemma}\label{lem:diag1}
Let $V$ be a matrix which is $S$--equivalent to a Seifert matrix of the knot $K$ and such that $V$ satisfies \eqref{equ:vvt}.
Let $A_K(t)$ be the matrix as in \eqref{equ:akt}. Then the matrix $A(t)=A_K(t)\oplus(1)$
$($i.e. the block diagonal sum of the matrices $A_K(t)$ and $(1))$
represents $\l(K)$ and the bilinear matrix $A(1)$ is diagonalizable over $\mathbb{Z}$.
\end{lemma}

\begin{proof}
Let $V$ be a Seifert matrix of the knot $K$ of size $2k$ satisfying \eqref{equ:vvt}.
Then we can write
\[ V=\bp B & C+I \\ C^t & D \ep \]
where $B$, $C$ and $D$ are $k\times k$ matrices, $I$ is the identity matrix and moreover  $B=B^t$, $D=D^t$. It is easy to compute that
\[A_K(1)=\bp B & I \\ I & 0\ep.\]
It is straightforward to verify that $A_K(1)$ is congruent over $\mathbb{Q}$
to the block sum of $I$ and $-I$, hence $A_K(1)$, viewed as a symmetric bilinear pairing,  is indefinite. If we consider $A(t)=A_K(t)\oplus(1)$ (which
clearly represents the same Blanchfield pairing as $A_K(t)$),
then $A(1)$  is an indefinite, odd symmetric bilinear pairing over $\mathbb{Z}$, hence by \cite[Theorem 4.3]{HM73} it is diagonalizable.
\end{proof}

\subsection{Definition of $n(K)$}

Let $K\subset S^3$ be a knot.
It follows from Lemma \ref{lem:diag1} that it makes sense to  define $n(K)$ as  the minimal size of a hermitian matrix $A$ over $\zt$ such that
\begin{itemize}
\item $\lambda(A)\cong \lambda(K)$;
\item the matrix $A(1)$ is congruent over $\Z$ to a diagonal matrix which has $\pm 1$'s on the diagonal.
\end{itemize}
In fact we can use Lemma \ref{lem:diag1} to deduce a more precise statement.

\begin{lemma}\label{lem:ndegdelta}
For any knot $K$ we have the following inequality
\[ n(K)\le \deg\Delta_K(t)+1.\]
\end{lemma}

\begin{proof}
It is well--known, see e.g. \cite[p.~195]{Lev70} that any Seifert matrix is S--equivalent to a matrix $V$ which is non--singular
and which satisfies (\ref{equ:vvt}).
Since $\det(Vt-V^t)=\Delta_K(t)$ it follows easily that $V$ is a matrix of size $\deg \Delta_K(t)$.
The corollary now follows immediately from Lemma   \ref{lem:diag1}.
\end{proof}

\begin{remarks}
\bn
\item
Suppose that
\[ V=\bp B & C+I \\ C^t & D \ep \]
is a matrix of size $\deg \Delta_K(t)$  which is $S$--equivalent to a Seifert matrix of $K$ and $B=B^t$, $D=D^t$.
If $B$ itself represents an odd pairing, then $A_K(1)$ is already diagonalizable. In that case $n(K)\le\deg\Delta_K(t)$.
\item Fogel \cite[Section~3.3]{Fo93}  gives examples of two knots $K_1$ and $K_2$ such that $n(K_1\# K_2)=n(K_1)=n(K_2)=1$. This shows that the $n(K)$ invariant
is in general not additive.
This is in contrast
to the conjecture that the unknotting number is additive
(see \cite[Problem~1.69\,(B)]{Kir97} and see \cite{Sch85} for some strong evidence towards this conjecture).
\en
\end{remarks}

\subsection{The Blanchfield pairing and intersection pairings on 4--manifolds}

We now turn to the proof that $u_a(K)\geq n(K)$. We will show that the 0--framed surgery on a knot which can be turned into an Alexander polynomial one knot using $u_+$ positive and $u_-$ negative crossing changes
cobounds a 4--manifold with certain properties. We will then show that a matrix
representing the equivariant intersection pairing on that 4--manifold gives in fact a presentation matrix for the Blanchfield pairing of $K$.

Given a knot $K\subset S^3$ we denote in the following by $M(K)$ the 0--framed surgery on $K$.
Furthermore, given a topological 4--manifold $W$ with boundary $M$,  we consider the following sequence of maps
\[ H_2(W;\Z)\xrightarrow{\i} H_2(W,M;\Z)\xrightarrow{\PD} H^2(W;\Z)\xrightarrow{\ev} \hom_\Z(H_2(W;\Z)),\]
where $\i$ denotes the inclusion induced map, $\PD$ denotes Poincar\'e duality and $\ev$ denotes the evaluation map.
This defines a pairing \[ H_2(W;\Z)\times H_2(W;\Z)\to \Z,\] called the \emph{ordinary intersection pairing of $W$}, which is well--known to be symmetric.
In the following we will several times make implicit use of the following lemma.

\begin{lemma}
Suppose the following hold:
\bn
\item $M$ is connected,
\item  $H_1(M;\Z)\to H_1(W;\Z)$ is an isomorphism,
\item $H_1(W;\Z)$ is torsion--free,
\en
then the ordinary intersection pairing is non--singular.
\end{lemma}

\begin{proof}
The assumption that $H_1(M;\Z)\to H_1(W;\Z)$ is an isomorphism implies by Poincar\'e duality that
$H^2(M;\Z)\to H^3(W,M;\Z)$ is an isomorphism. From the universal coefficient theorem
it follows that $\hom_\Z(H_2(M;\Z),\Z)\to \hom_\Z(H_3(W,M;\Z),\Z)$ is an isomorphism. But $H_2(M;\Z)\cong H^1(M;\Z)$
and $H_3(W,M;\Z)\cong H^1(W;\Z)\cong \hom(H_1(W;\Z),\Z)$ are torsion--free,
it thus follows
 that $H_3(W,M;\Z)\to H_2(M;\Z)$ is an isomorphism.
It follows from the long exact sequence of the pair $(W,M)$ that the map $\i\colon H_2(W;\Z)\to H_2(W,M;\Z)$ is an isomorphism.
The assumption that $H_1(W;\Z)$ is torsion--free implies by the universal coefficient theorem that the evaluation map $\ev\colon H^2(W;\Z)\to \hom_\Z(H_2(W;\Z),\Z)$ is an isomorphism.
It now follows that the ordinary intersection pairing is non--singular.
\end{proof}

We now consider a topological 4--manifold $W$ with boundary $M$ such that $\pi_1(W)=\Z$.
We then  consider the following sequence of maps
\be \label{equ:lambdaw} H_2(W;\L)\xrightarrow{\i} H_2(W,M;\L)\xrightarrow{\PD} \ol{H^2(W;\L)}\xrightarrow{\ev} \ol{\hom_{\L}(H_2(W;\L),\L)},\ee
where the first map is again the inclusion induced map, the second map is Poincar\'e duality and the third map is the evaluation map. This composition of maps
defines a  pairing
\[ H_2(W;\L)\times H_2(W;\L)\to \L,\]
which is well-known to be hermitian. We refer to this pairing as the \emph{twisted intersection pairing} on $W$. Now we shall introduce
a following notion, which we shall use several times in the future.

\begin{defn}\label{def:tamelycobounds}
Let $K$ be a knot and $M(K)$ the zero framed surgery on $K$. We shall say that a four manifold $W$ \emph{tamely cobounds $M(K)$} if the
following conditions are satisfied:
\bn
\item $\partial W=M(K)$,
\item the inclusion induced map $H_1(M(K);\Z)\to H_1(W;\Z)$ is an isomorphism,
\item $\pi_1(W)=\Z$,
\en
If furthermore the intersection form on $H_2(W;\Z)$ is diagonalizable, we say that $W$ \emph{strictly cobounds $M(K)$}.
\end{defn}

The following theorem will be the key ingredient in the proof that $u_a(K) \geq n(K)$.

\begin{theorem}\label{thm:w1}
Let $K$ be a knot. Suppose there exists a topological 4--manifold $W$, which tamely cobounds $M(K)$.
Then $H_2(W;\L)$ is free of rank $b_2(W)$. Furthermore, if $B$ is an integral matrix representing the ordinary intersection pairing of $W$,
then there exists a basis $\BB$ for $H_2(W;\L)$ such that the matrix $A(t)$ representing the twisted intersection pairing  with respect
to $\BB$ has the following two properties:
\bn
\item $\l(A(t))\cong \l(K)$,
\item $A(1)=B$.
\en
\end{theorem}

The proof of Theorem \ref{thm:w1} is rather long and will require all of the following section.

\subsection{Proof of Theorem \ref{thm:w1}}

Let $K$ be a knot and let $W$ be a topological 4--manifold $W$, which tamely cobounds $M(K)$.
Throughout this section we write  $M:=M(K)$.
We first want to prove the following lemma:

\begin{lemma}
The $\L$--module $H_2(W;\L)$ is free of rank $b_2(W)$.
\end{lemma}

\begin{proof}
We first want to show that $H_2(W;\L)$ is a free $\L$-module.
Note that $H_2(W;\L)$ is a finitely generated $\L$-module since $\L$ is Noetherian.
By \cite[Corollary~3.7]{Ka86} the module $H_2(W;\L)$ is free if and only if
$\ext_{\L}^i(H_2(W;\L),\L)=0$ for $i=1,2$.

Note that $\pi_1(W)\cong \Z$ implies that $H_1(W;\L)=0$.
We also have $H_4(W;\L)=0$. We furthermore have
an isomorphism $H_0(M;\L)\to H_0(W;\L)$. We thus conclude from the long exact homology sequence corresponding to the
pair $(W,M)$ that
$H_0(W,M;\L)=0$ and $H_1(W,M;\L)=0$.

Recall that the UCSS (see Section \ref{section:basics}) starts with  $E_{p,q}^2=\ext^p_{\L}(H_q(W;\L),\L)$ and converges
to $H^*(W;\L)$. Furthermore the differentials have degree $(1-r,r)$.
By the above we have $E_{p,q}^2=0$ for $q=1$ and $q=4$. Since $\L$ has cohomological dimension 2 we also
have $E_{p,q}^2=0$ for $p\geq 3$. Finally note that
\[ E_{2,0}^2=\ext^2_{\L}(H_0(W;\L),\L)=\ext^2_{\L}(\L/(t-1)\L,\L)=0.\]
It now follows from the UCSS that we have a monomorphism
\[  E_{1,2}^2=\ext^1_{\L}(H_2(W;\L),\L)\to H^3(W;\L).\]
But $H^3(W;\L)\cong \ol{H_1(W,M;\L)}=0$.
Similarly, it follows  from the UCSS that we have a monomorphism
\[ E^2_{2,2}=\ext^2_{\L}(H_2(W;\L),\L) \to H^4(W;\L).\]
But $H^4(W;\L)\cong \ol{H_0(W,M;\L)}=0$. This concludes the proof of the claim that $H_2(W;\L)$ is a free module.

We now turn to the proof that $H_2(W;\L)$ is a free $\L$-module of rank $s:=b_2(W)$.
It remains to show that $H_2(W;\L)$ is of rank $s$. Since $\O$ is flat over $\L$, it suffices to show that
$\dim_\O(H_2(W;\O))=s$. It is clear that $H_i(W;\O)=0$ for $i=0,1,4$. Furthermore
$ H_3(W;\O)\cong \ol{ H^1(W,M;\O)}$. But since $\O$ is a field the latter is isomorphic
to $\ol{H_1(W,M;\O)}$ which is zero.
We thus calculate
\[ \dim_\O(H_2(W;\O))=\sum_{i=0}^4 (-1)^i\dim_\O(H_i(W;\O))=\chi(W).\]
Now note that $b_0(W)=b_1(W)=1$ and $b_4(W)=0$. Also note that $H^3(W;\Z)=H_1(W,M;\Z)=0$ since we assume that
$H_1(M;\Z)\to H_1(W;\Z)$ is an isomorphism.
It thus follows that $b_3(W)=0$, and we see that $\chi(W)=b_2(W)=s$.
This concludes the proof of the lemma.
\end{proof}

We now write $s=b_2(W)$. We pick a basis $\BB$ for $H_2(W;\L)$ and denote by $A=A(t)$ the corresponding $s\times s$--matrix representing the twisted intersection pairing.
 Note that $A$ is a hermitian $s\times s$-matrix.
By the argument of \cite[Lemma~2.2]{FHMT07} we see that the matrix $A(1)$ represents the ordinary intersection pairing on $H_2(W;\Z)$.
In particular there exists an integral matrix $P$ such that  $PA(1)P^t=B$. After acting on the basis $\BB$ by the matrix $P$
we can without loss of generality assume that in fact $A(1)=B$.
The following lemma now concludes the proof of Theorem \ref{thm:w1}.

\begin{lemma}\label{lem:lambdaak}
The pairing $\lambda(A)$ is isometric to $\lambda(K)$.
\end{lemma}

The proof of the lemma will require the remainder of this section.
We first want to prove the following claim:

\begin{claim}
The following is a short exact sequence:
\be \label{equ:ses} 0\to H_2(W;\L)\to H_2(W,M;\L)\to H_1(M;\L)\to 0.\ee
\end{claim}

\begin{proof}
To prove the claim we first consider the following  exact sequence
\[ H_2(M;\L)\to H_2(W;\L)\to H_2(W,M;\L)\to H_1(M;\L)\to H_1(W;\L)\to \dots\]
Recall  that $H_1(W;\L)=0$. Also note that
\[ H_2(M;\L)\otimes_{\L}\O \cong H_2(M;\O)\cong \ol{H^1(M;\O)}\cong \ol{\hom_{\O}(H_1(M;\O),\O)}=0\]
since $H_1(M;\O)=H_1(M;\L)\otimes_{\L}\O=0$ (here we used that $H_1(M;\L)$ is torsion).
In particular $H_2(M;\L)$ is torsion and the map $ H_2(M;\L)\to H_2(W;\L)$ is trivial since $H_2(W;\L)$ is a free $\L$-module.
This now concludes the proof of the claim.
\end{proof}

We now define a Blanchfield pairing on $H_1(M;\L)$ and an intersection pairing on $H_2(W,M;\L)$.
First of all, similar to (\ref{equ:defbl}) we can consider the following sequence of isomorphisms:
\[H_1(M;\L)\xrightarrow{\PD} \ol{H^2(M;\L)}\xleftarrow{\cong} \ol{H^1(M;\O/\L)}
\xrightarrow{\ev} \hom_{\L}(H_1(M;\L),\O/\L).\]
This defines a hermitian non-singular pairing
\be\label{equ:blanonM}
 H_1(M;\L)\times H_1(M;\L)\to \O/\L.
\ee
It is well-known that the natural map $H_1(X(K);\L)\to H_1(M;\L)$ is an isomorphism,
and it follows immediately that the Blanchfield pairing
on $X(K)$ is isometric to the pairing \eqref{equ:blanonM} on $M$.

Secondly, we consider the following sequence of maps
\be \label{equ:wm}
\ba{rcl} H_2(W,M;\L)&\xrightarrow{\PD}& \ol{H^2(W;\L)}\to \ol{H^2(W;\O)}\cong \ol{H^2(W,M;\O)}\\
&\xrightarrow{\ev}& \ol{\hom(H_2(W,M;\L),\O)}.\ea\ee
Here, for the third map we made use of the fact that $H_1(M;\L)$ is $\L$-torsion, therefore \eqref{equ:ses}
implies that the inclusion induced map
$ H^2(W,M;\O)\to H^2(W;\O)$
is an isomorphism. The other maps in (\ref{equ:wm}) are given by Poincar\'e duality, inclusion of rings and the evaluation homomorphism.
The sequence of maps in (\ref{equ:wm}) now defines a hermitian pairing
\[ H_2(W,M;\L)\times H_2(W,M;\L)\to \O.\]

\begin{claim}
The intersection pairing on $W$, the intersection pairing on $H_2(W,M;\L)$
and the Blanchfield pairing on $M$ fit into the following commutative diagram, where the left vertical maps form a short exact sequence:
\be \label{equ:forms} \xymatrix{ H_2(W;\L)\times H_2(W;\L) \ar[d]\ar[r] & \L\ar[d] \\
H_2(W,M;\L)\times H_2(W,M;\L) \ar[d]\ar[r] & \O\ar[d] \\
H_1(M;\L)\times H_1(M;\L) \ar[r] & \O/\L.}\ee
\end{claim}

\begin{proof}
In the previous claim we already showed that the left vertical maps form a short exact sequence.
We now consider the following diagram
\[
\xymatrix{ H_2(W;\L)\times H_2(W;\L) \ar[d]\ar[r] & \L\ar[d] \\
H_2(W;\O)\times H_2(W;\O) \ar[d]\ar[r] & \O\ar[d] \\
H_2(W,M;\O)\times H_2(W,M;\O) \ar[r] & \O \\
H_2(W,M;\L)\times H_2(W,M;\L) \ar[u]\ar[r] & \O.\ar[u]}\]
The pairings on $\O$--homology are defined in complete analogy to the corresponding pairings on $\L$--homology,
and the vertical maps are the obvious maps. It now follows easily from the definitions
that this is a commutative diagram. Since the image of $H_2(W;\L)\to H_2(W,M;\O)$
lies in the image of $H_2(W,M;\L)\to H_2(W,M;\O)$ it now follows  that the top square in the diagram of the claim commutes.

We now consider the following diagram
\be\label{equ:c}
\xymatrix{
H_2(W,M;\L)\ar[r]\ar[d]&H_1(M;\L)\ar[d]\\
\ol{H^2(W;\L)}\ar[d]&\ol{H^2(M;\L)}\\
\hom(H_2(W,M;\L),\O)\ar[d]&\ol{H^1(M;\O/\L)}\ar[u]^{\cong}\ar[d]\\
\ol{\hom(H_2(W,M;\L),\O/\L)}&\ar[l]\ol{\hom(H_1(M;\L),\O/\L)}}\ee
where the left middle vertical map is a part of the definition of the intersection pairing on $H_2(W,M;\L)$,
furthermore the horizontal maps are the maps induced by long exact sequences corresponding to the pair $(W,M)$.
By \cite[Section~6]{Lei06} this diagram commutes. This now implies that the lower square in the claim also commutes.
\end{proof}

\begin{claim}
The evaluation map
\[ H^2(W;\L)\xrightarrow{\ev}\hom_{\L}(H_2(W;\L),\L)\]
is an isomorphism.
\end{claim}

\begin{proof}
In order to prove the claim we have to  study the UCSS corresponding to $H^2(W;\L)$.
Note that $\ext_{\L}^1(H_0(W;\L),\L)=\L/(t-1)\L$ is $\L$-torsion, hence the differential
\[d_2\colon E^2_{1,0}= \ext_{\L}^1(H_0(W;\L),\L)\to E^2_{0,2}=\ext_{\L}^0(H_2(W;\L),\L)\]
is zero since $ \ext_{\L}^0(H_2(W;\L),\L)=\hom_{\L}(H_2(W;\L),\L)$ is $\L$-torsion free.
It now follows (using the earlier discussion) that the UCSS for $H^2(W;\L)$ gives rise to the desired isomorphism
\be \label{equ:h2} H^2(W;\L)\xrightarrow{\cong} \ext^0_{\L}(H_2(W;\L),\L)=\hom_{\L}(H_2(W;\L),\L).\ee
\end{proof}

Recall that we picked a basis $\BB$ for $H_2(W;\L)$ and that we denote by $A=A(t)$ the corresponding matrix representing the twisted intersection pairing on $H_2(W;\L)$.
Now note that by Poincar\'e duality and by the above claim we have two isomorphisms
\be \label{equ:bc} H_2(W,M;\L)\underset{\cong}{\xrightarrow{\PD}}\ol{ H^2(W;\L)}\underset{\cong}{\xrightarrow{\ev}}\ol{\hom_{\L}(H_2(W;\L),\L)}.\ee
We now  endow $H_2(W,M;\L)$ with the basis $\CC$ which is dual to $\BB$.
It follows easily from (\ref{equ:lambdaw}) and (\ref{equ:bc})  that the inclusion induced map $H_2(W;\L)\to H_2(W,M;\L)$
with respect to the bases $\BB$ and $\CC$ is given by $A$.

We now rewrite the diagram (\ref{equ:forms}) in terms of our bases, we thus obtain the following diagram
\[  \xymatrix{ \L^s \times \L^s \ar[d]^{(v,w)\mapsto (Av,Aw)}\ar[rrr]^{(v,w)\mapsto \ol{v}^tAw} &&&\L\ar[d] \\
\L^s \times \L^s \ar[rrr]^{(v,w)\mapsto \ol{v}^tA^{-1}w}\ar[d]&&&\O \ar[d]\\
H_1(M;\L)\times H_1(M;\L) \ar[rrr] &&& \O/\L.}\]
The statement of Lemma \ref{lem:lambdaak}  now follows from  this diagram and the fact that the left vertical maps form a short exact sequence.
This concludes the proof of Theorem \ref{thm:w1}.

\subsection{Proof of Theorem \ref{thm1}}

Clearly the following theorem, combined with Theorem \ref{thm:w1} implies
  Theorem \ref{thm1} from the introduction.

\begin{theorem}\label{thm:w2}
Let $K$ be a knot such that  $u_+$ positive crossing changes and $u_-$ negative crossing changes turn $K$ into an Alexander polynomial one knot $J$.
Then there exists
 an oriented topological 4--manifold $W$ which strictly cobounds $M(K)$. Moreover,
the intersection pairing on $H_2(W;\Z)$ is represented by a diagonal matrix of size $u_++u_-$
such that $u_+$ entries are equal to $-1$ and $u_-$ entries are equal to $+1$.
\end{theorem}

\begin{proof}
We first recall the following well known reinterpretation of a crossing change.
Let $K\subset S^3$ be a knot. Suppose we perform an $\eps$--crossing change along a crossing.
We denote by $D\subset S^3$ an embedded disk which intersects $K$ in precisely two points with opposite orientations,
one point on each strand involved in the crossing change.
If we now perform $\eps$--surgery on the curve $c$, then the resulting 3--manifold $\S$ is diffeomorphic to $S^3$, and $K\subset \S$
is the result of performing an $\eps$--crossing change.

In the following we will use  the following notation: let $c_1,\dots,c_s$ be simple closed curves which form the unlink in $S^3$ and let $\eps_1,\dots,\eps_s\in \{-1,1\}$, then we denote by
$\S(c_1,\dots,c_s,\eps_1,\dots,\eps_s)$ the result of performing $\eps_i$--surgery along $c_i$ for $i=1,\dots,s$.
Note that this 3--manifold is diffeomorphic to the standard 3--sphere.

Let $K$ be a knot such that  $u_+$ positive crossing changes and $u_-$ negative crossing changes turn $K$ into an Alexander polynomial one knot $J$.
Put differently, there exists an Alexander polynomial one knot $J$
such that  $u_+$ negative  crossing changes and $u_-$ negative positive changes turn $J$ into $K$.
We write $s=u_++u_-$ and $n_i=-1$ for $i=1,\dots,u_+$ and $n_i=1$ for $i=u_++1,\dots,u_++u_-$.
By the above discussion  there exist simple  closed curves  $c_1,\dots,c_s$  in $X(J)$
with the following properties:
\bn
\item $c_1,\dots,c_s$ are the unlink in $S^3$,
\item the linking numbers $\lk(c_i,K)$ are zero,
\item the image of $J$ in
\[ \S(c_1,\dots,c_s,n_1,\dots,n_{s})\]
is the knot $K$.
\en
Note that the curves $c_1,\dots,c_s$ lie in $S^3\sm \nu J$ and we can thus view them as lying in $M(J)$.
The manifold  $M(K)$ is then the result of $n_i$ surgery on $c_i\subset M(J)$ for $i=1,\dots,s$.

Since $J$ is a knot with trivial  Alexander polynomial it follows from Freedman's theorem (see \cite[Theorem 117B]{FQ90}),
that $J$ is topologically slice, in fact there exists a locally flat slice disk $D\subset D^4$ for $J$ such that
$\pi_1(D^4\sm D)=\Z$. We now write $X:=D^4\sm \nu D$. Then $X$ is an oriented topological 4--manifold $X$ with  the following properties:
\bn
\item $\partial X=M(J)$ as oriented manifolds,
\item $\pi_1(X)=\Z$,
\item $H_1(M(J);\Z)\to H_1(X;\Z)$ is an isomorphism,
\item $H_2(X;\Z)=0$.
\en
We denote by $W$ the 4-manifold which is the result of  adding 2-handles along
$c_1,\dots,c_s\subset M(J)$ with framings $n_1,\dots,n_s$ to $X$.
Note that $\partial W=M(K)$ as oriented manifolds.
We will henceforth write $M=M(K)$.
Note that the curves $c_1,\dots,c_s$ are null-homologous, it follows easily that the map
$H_1(M;\Z)\to H_1(W;\Z)$ is an isomorphism and that $\pi_1(W)\cong \Z$.
It thus remains to prove the following claim:

\begin{claim}
The ordinary intersection pairing on $W$ is represented by a  diagonal matrix of size $u_++u_-$ such that $u_+$ diagonal entries are equal to $-1$ and $u_-$ diagonal entries are equal to $1$.
\end{claim}

Recall that the curves $c_1,\dots,c_s$ form the unlink in $S^3$ and that the linking numbers $\lk(c_i,J)$ are zero.
In particular the curves $c_1,\dots,c_s$ are also null homologous in $M(J)$.
It is clear that we can now find disjoint surfaces $F_1,\dots,F_s$ in $M(J)\times [0,1]$
such that $\partial F_i=c_i\times 1$. By adding the cores of the 2--handles attached to the $c_i$
we now obtain closed surfaces $C_1,\dots,C_s$ in $W$. It is straightforward to see that
$C_i\cdot C_j=0$ for $i\ne j$ and $C_i\cdot C_i=n_i$.
A Meyer--Vietoris argument shows that the surfaces $C_1,\dots,C_s$ present a basis for $H_2(W;\Z)$.
In particular the intersection matrix on $W$ with respect to this basis is given by $(C_i\cdot C_j)$,
i.e. it  is a diagonal matrix such that $u_+$ diagonal entries are equal to $-1$ and $u_-$ diagonal entries are equal to $1$.
This concludes the proof of the claim.

\end{proof}

\begin{remark}
In the proof of Theorem  \ref{thm:w2} (and thus in the proof that $u_a(K)\geq n(K)$),
we made use of Freedman's theorem that a knot with trivial Alexander polynomial is topologically slice.
This deep topological fact is not necessary to prove Theorem~\ref{thm1}, but it simplifies the algebra and the exposition.
\end{remark}

If a knot $K$ has unknotting number $u$, then
 Montesinos \cite{Mo73} has shown that the 2--fold branched cover $\S(K)$  is given by Dehn surgery on some framed
link in $S^3$ with $u$ components, with half-integral framing coefficients.
This fact is used in the original proof of the Lickorish obstruction
and it lies at the heart of some of the deepest results on unknotting numbers (see e.g. \cite{OS05} and \cite{Ow08})
which are obtained by studying Heegaard--Floer invariants of the \emph{compact} 3--manifold $\S(K)$.

Let $K$ be a knot such that  $u_+$ positive crossing changes and $u_-$ negative crossing changes turn $K$ into the unknot.
If we take $X=S^1\times D^3$ in the proof of Theorem \ref{thm:w2}, then we immediately see that
 there exists  an oriented \emph{smooth} 4--manifold $W$ which satisfies the Properties (1) -- (4) of Theorem \ref{thm:w2}.
This suggests  that further information on unknotting numbers can be obtained from considering higher cyclic covers (or the infinite cyclic cover)
of $M(K)$.

\section{The Blanchfield pairing and the linking pairing}\label{section:linkingform}

In this section we will relate the Blanchfield pairing to the linking pairing on the homology of the 2--fold branched cover of a given knot $K$.

\subsection{Preliminary results}
The following proposition is a key tool in relating $n(K)$ to other invariants:

\begin{proposition}[\expandafter{\cite[Proposition~1.7.1]{Ran81}}] \label{prop:ra81}
Let $A$ and $B$ be hermitian matrices over $\zt$ with $\det(A(1))=\det(B(1))=\pm 1$.
Then $\lambda(A)\cong \lambda(B)$  if and only if $A$ and $B$ are related by a sequence of  the following three moves:
\bn
\item replace $C$ by $PC\ol{P}^t$ where $P$ is a matrix over $\zt$ with $\det(P)=\pm 1$,
\item replace $C$ by the block sum $C\oplus D$ where $D$ is a hermitian matrix over $\zt$ with $\det(D)=\pm 1$,
\item the inverse of (2).
\en
\end{proposition}

We can now prove the following lemma:

\begin{lemma}\label{lem:signetaa}
Let $A(t)$ be a hermitian matrix over $\zt$ with $\l(A(t))\cong \l(K)$, then
\[ \ba{rcll} \sign(A(z))-\sign(A(1))&=&\s_z(K), &\mbox{ for any $z\in S^1$ and }\\
\null(A(z))&=&\eta_z(K),&\mbox{ for any $z\in \C\sm \{0,1\}$.}\ea \]
\end{lemma}

\begin{proof}
First let $D(t)$ be any hermitian matrix over $\zt$. It is well--known that the function
\[ \ba{rcl} S^1&\to & \Z \\ z&\mapsto &\sign(D(z))\ea \]
is constant
outside of the set of zeros of $D(t)$. In particular if $\det(D(t))=\pm 1$, then the signature function is constant.
It now follows easily from  Proposition \ref{prop:blakt} that if $A(t)$ and $B(t)$ are hermitian matrices over $\zt$ with $\l(A(t))\cong \l(B(t))$,
then
\[ \sign(A(z))-\sign(A(1))=\sign(B(z))-\sign(B(1)) \mbox{ for any $z\in S^1$.}\]
The  first claim now follows  from (\ref{equ:aktlt}) and Proposition \ref{prop:blakt}.
The proof of the second statement also follows from a similar argument.
\end{proof}

\subsection{Linking pairings}

We will now relate the Blanchfield pairing to the linking pairing on the 2--fold branched cover. Later on, this will allow us to relate $n(K)$ to the Lickorish obstruction and the Jabuka obstruction,
as well as
to get new computable obstructions to $n(K)=2$ or $n(K)=3$.

A \emph{linking pairing} is a non-singular symmetric bilinear pairing $H\times H\to \Q/\Z$ where $H$ is a finite abelian group
of odd order. If $l$ and $l'$ are isometric linking pairings, then we write $l\cong l'$.
An example is the linking pairing $l(K)$ defined on $H_1(\Sigma(K))$.

 Given a symmetric integral matrix $A$ with $\det(A)$ odd we denote by $l(A)$  the linking pairing which is defined as follows:
\[ \ba{rcl} \Z^n/A\Z^n \times \Z^n/A\Z^n&\to &\Q/\Z \\ (v,w)&\mapsto &v^t A^{-1}w.\ea \]
Given a linking pairing $l\colon H\times H\to \Q/\Z$ and $n\in \Z$, coprime to $|H|$, we denote by $n\cdot l$ the
linking pairing given by $(n\cdot l)(v,w):=n\cdot l(v,w)$.

We can now formulate and prove the following lemma.

\begin{lemma}\label{lem:atlinkingform}
Let $K$ be a knot and let $A(t)$ be a hermitian matrix over $\zt$ such that $\l(A(t))\cong \l(K)$.
Then
\[ l(A(-1)) \cong 2 l(K).\]
\end{lemma}

\begin{proof}
We now   denote by $\MM$ the set of all hermitian matrices $A$ over $\zt$ such that $\det(A(1))=\pm 1$.
 We say that $A,B\in \MM$ are \emph{equivalent}, written $A\sim B$, if $\l(A)$ and $\l(B)$ are isometric.
 We furthermore denote by $\LL$ the set of isometry classes of linking pairings.
We consider the  map
\[ \ba{rcl} \Phi\colon  \MM&\to & \LL \\
A(t)&\mapsto &l(A(-1)).\ea \]
(Note that $\det(A(-1))\equiv \det(A(1))=\pm 1 \,\mod 2$.)
It follows immediately from Proposition \ref{prop:ra81}
that the map $\Phi$  descends to a map
\[ \MM/\sim \,\,\to \,\LL.\]

Let $V$ be a Seifert matrix  for $K$.  We define $A_K(t)$ as in Section \ref{section:seifertbf}.
It is well-known (see e.g. \cite{Go78}) that the  linking pairing $l=l(K)$ on $H_1(\Sigma(K);\Z)$  is isometric to
$l(V+\vt)$.

An argument analogous to the proof of Proposition \ref{prop:blakt} with $\Lambda$ replaced by $\Z$, $\L_0$ replaced by $\Z[\frac{1}{2}]$
and $t$ replaced by $-1$ then shows that the linking pairing $\Phi(A_K(t))$ is isometric to the pairing $2l(K)$.

Now let $A(t)$ be a hermitian matrix over $\zt$ such that $\l(A(t))\cong \l(K)$. Then $A(t)\sim A_K(t)$ and it follows from the above
that
\[ l(A(-1))=\Phi(A(t))\cong \Phi(A_K(t))\cong 2l(K).\]
\end{proof}

\subsection{The linking pairing and the algebraic unknotting number}

In the following we refer to a positive definite matrix as $(+1)$--definite and we refer to a negative definite matrix as a $(-1)$--definite matrix.
We now have the following theorem.

\begin{theorem}\label{thm:representlk}
If $n(K)=n$, then there exists a symmetric $n\times n$--matrix $A$ over $\Z$
which has  the following three properties:\
\bn
\item $|\det(A)|=|\det(K)|$,
\item $l(A)\cong 2l(K)$,
\item $A$ modulo two equals the identity matrix.
\en
If $\s(K)=2n\cdot \eps$ with  $\eps \in \{-1,1\}$, then we can furthermore arrange that $A$ has the following two properties:
\bn
\item[(4)] $A$ is $\eps$--definite,
\item[(5)] the diagonal entries of $A$ modulo four are equal to $-\eps$.
\en
\end{theorem}

\begin{remark}
Theorem \ref{thm:representlk} is closely related to \cite[Theorem~3]{Ow08}, which is the main technical theorem of \cite{Ow08}.
More precisely, Owens shows in  \cite[Theorem~3]{Ow08} that if $u(K)=u$, then $\S(K)$ can be obtained by Dehn surgery along a $u$--component link with a certain framing matrix.
It then follows from that surgery description of $\S(K)$ that there exists a $u\times u$--matrix $A$ over $\Z$ which has properties described in Theorem \ref{thm:representlk}.
\end{remark}

\begin{proof}
Since $n(K)=n$ we can find  a hermitian $n\times n$--matrix $B(t)$ over $\zt$ such that $\l(B)\cong \l(K)$
and such that $B(1)$ is diagonalizable over $\Z$. Note that $B(t)$ in particular represents the Alexander module,
it follows that $\det(B(t))=\pm \Delta_K(t)$, i.e. $\det(B(1))=\pm 1$ and $\det(B(-1))=\pm \det(K)$.

Let $P$ be an invertible integral matrix such that $PB(1)P^t$ is diagonal. After replacing $B$ by $PBP^t$ we can thus arrange
that $B(1)$ is already diagonal. We denote the diagonal entries by $\eps_1,\dots,\eps_n$.
We furthermore denote by $b_{ij}=b_{ij}(t)$ the entries of $B=B(t)$. The fact that $b_{ii}(1)=\eps_i$ and the fact that $b_{ii}(t^{-1})=b_{ii}(t)$ implies that
\[ b_{ii}=\eps_i+(t-1)(t^{-1}-1)c_{ii}\]
for some polynomial $c_{ii}\in \zt$ with $c_{ii}(t^{-1})=c_{ii}(t)$.
Furthermore, given $i\ne j$, the fact that $b_{ij}(1)=0$ implies that
\[ b_{ij}=(t-1)\cdot c_{ij}\]
for some $c_{ij}\in \zt$.
By Lemma \ref{lem:atlinkingform} the matrix $A:=B(-1)$ represents $2l(K)$. By the above we have $\det(A)=\pm \det(K)$.
It follows immediately from the above that $A=B(-1)$  agrees with the identity matrix modulo two.

We now assume that $\s(K)=2n\cdot \eps$ with $\eps \in \{-1,1\}$. It follows from Lemma  \ref{lem:signetaa}
that $\sign(B(-1)\oplus -B(1))=\sign(B(-1))-\sign(B(1))=2n\cdot \eps$. Since $B$ is an $n\times n$--matrix this implies that $\sign(B(-1))=n\cdot \eps$ and $\sign(B(1))=-n\cdot \eps$.
In particular $A=B(-1)$ is $\eps$--definite.  Since $B(1)$ is $\eps$--definite it follows also that
 $\eps_i=-\eps$ for $i=1,\dots,n$.  Since $b_{ii}=\eps+(t-1)(t^{-1}-1)c_{ii}$ it now follows that  $A$ has the desired fourth property.
\end{proof}

\section{Comparison of classical invariants}\label{section:classicallower} \label{section:classicaldef}

In this section we show that $n(K)$ subsumes the classical invariants stated in Theorem~\ref{thm:subsume}. We discuss each
of the criteria of Theorem~\ref{thm:subsume} in a separate subsection.

\subsection{Lower bounds on $u_a(K)$: The Nakanishi index}\label{section:nakanishi}

Let $K$ be a knot.
The first lower bounds on the unknotting number $u(K)$ were given by Wendt \cite{We37} who showed that
 \[ u(K)\geq \mbox{minimal number of generators of $H_1(\S(K);\Z)$}.\]
 (See also \cite{Kin57,Kin58} and \cite[Section~E]{BW84} for further details).
 These lower bounds are subsumed by the Nakanishi index.  More precisely,   by \cite[Theorem~3]{Na81} (see also \cite[Theorem~11.5.1]{Ka96}) we have the following inequality:
\[ u(K)\geq  m(K).\]
It is clear that if a hermitian matrix $A$ over $\zt$ satisfies $\l(K)\cong \l(A)$,
then $A$ is also presentation matrix for the Alexander module. We thus see that
\[ n(K)\geq m(K).\]
Together with Theorem \ref{thm1} this implies that $m(K)$ gives in fact a lower bound on the algebraic unknotting number.
This can also be deduced from modifying the proof provided by Nakanishi.

\subsection{Lower bounds on $u_a(K)$: The Levine--Tristram signatures and the nullities}\label{section:lteta}
Levine--Tristram signatures are well--known to give lower bounds on the topological 4--genus
$g_4^{top}(K)$ of a knot, and hence lower bounds to the algebraic unknotting number.
(See \cite{Mus65}, \cite{Lev69}, \cite{Tri69} and \cite{Ta79} for details.)
But in fact the following stronger inequality holds:

\begin{theorem}\label{thm:atsignatures}
Let $K$ be a knot which can be turned into an Alexander polynomial one knot using $u_+$ positive crossing changes and $u_-$ negative crossing changes.
Then for any $z\in S^1$ we have
\[     -2u_-\leq      \eta_z(K)+\s_z(K)\leq 2u_+\]
in particular we have
\[ n(K) \geq \mu(K):=\frac{1}{2}\left(\max\{ \eta_z(K)+\s_z(K) \, |\, z\in S^1\}+\max\{ \eta_z(K)-\s_z(K) \, |\, z\in S^1\}\right).\]
\end{theorem}

We expect that this theorem is known to the experts, but we are not aware of a proof  in the literature.

\begin{proof}
By Theorem \ref{thm1} there exists a hermitian matrix $A(t)$ of size $u_++u_-$ over $\zt$ with the following two properties:
\bn
\item $A(1)$ is a diagonal matrix such that $u_+$ diagonal entries are equal to $-1$ and $u_-$ diagonal entries are equal to $1$,
\item $\l(A(t))\cong \l(K)$.
\en
Now let $z\in S^1$. We denote by $b_+$ (respectively $b_-,b_0)$ the number of positive (respectively negative, zero) eigenvalues of $A(z)$. Then
it follows from Lemma \ref{lem:signetaa} that
\[ \ba{rcl} \eta_z(K)+\s_z(K)&=&\null(A(z))+\sign(A(z))-\sign(A(1))\\
&=&b_0+(b_+-b_-)-(-u_++u_-)\\
&=&b_0+b_++b_--(-u_++u_-)-2b_-\\
&\leq& (u_++u_-)-(-u_++u_-)=2u_+.\ea \]
Similarly one shows that $ \eta_z(K)+\s_z(K)\geq -2u_-$.
\end{proof}

\subsection{Lower bounds on $u_a(K)$: The Livingston invariant}\label{section:liv}\label{section:livingston}

We first recall the definition of Livingston's invariant.
Let $S$
be a subring of $\Q(t)$. We denote by $Q$ its quotient field and we denote by $W(S\to Q)$ the Witt group of non-degenerate hermitian pairings over free $S$-modules which become non-singular after tensoring with $Q$.
Put differently, $W(S\to Q)$  is the Witt group of hermitian matrices over $S$ such that the determinant is a unit in $Q$.
We refer to \cite{HM73} and \cite{Ran81} for details.

Let  $K\subset S^3$ be a knot. We  define $\rho_S(K)$
to be the minimal size of a square matrix representing $A_K(t)$ in $W(S\to Q)$.
It follows from (\ref{equ:aktlt})  that for $S=\Q(t)$ we obtain Livingston's invariant $\rho(K)$ (see \cite{Liv11}).
It follows immediately from the definitions that
\[  \rho(K)=\rho_{\Q(t)}(K)\leq \rho_{\zt}(K)\]
for any knot $K$.
For the reader's convenience we will provide a proof to the following proposition, which is well known to the experts,

\begin{proposition}\label{prop:rhozt}
Let $K$ be a knot. Then
\[ 2g_4^{top}(K)\geq \rho_{\zt}(K).\]
\end{proposition}

We expect that $\rho_{\zt}(K)$ is the `best possible' lower bound on the topological 4--genus which can be obtained from the Seifert matrix.

\begin{proof}
Let $F$ be a Seifert surface of genus $k$ for $K$.
Denote by $g$ the topological 4-genus of $K$. In that case the argument provided in the appendix of \cite{Liv11} shows that  there exist $k-g$ linearly independent curves on $F$ on which the Seifert pairing vanishes.
Since the intersection pairing on $F$ is determined by the Seifert pairing it follows that the pairwise intersection
numbers of the curves are zero. In particular we can extend this set of linearly independent curves to a
symplectic basis on $H_1(F;\Z)$. The corresponding Seifert matrix $V$  now has the following two properties:
\[ V=\bp 0_{k-g\times k-g} & *_{k-g\times k-g} & *_{k-g\times 2g} \\ *_{k-g\times k-g} & *_{k-g\times k-g}&  *_{k-g\times 2g} \\  *_{2g\times k-g} & *_{2g\times k-g} & *_{2g\times 2g} \ep \mbox{ and }  V-\vt= \bp 0 & \id_k \\ -\id_k &0 \ep,\]
where the subscripts indicate the  size of the matrix.
It now follows that $A_K(t)$ (as defined in Section \ref{section:seifertbf}) is of the form
\[ A_K(t)=\bp 0_{k-g\times k-g} & *_{k-g\times k-g} & *_{k-g\times 2g} \\ *_{k-g\times k-g} & *_{k-g\times k-g}&  *_{k-g\times 2g} \\  *_{2g\times k-g} & *_{2g\times k-g} & *_{2g\times 2g} \ep.\]
It is well-known that one can find an invertible matrix $P$ over $\zt$ such that
\[ PA_K(t)P^t= \bp 0_{k-g\times k-g} & B_{k-g\times k-g} & 0_{k-g\times 2g} \\ B_{k-g\times k-g}^t & 0_{k-g\times k-g}&  0_{k-g\times 2g} \\  0_{2g\times k-g} & 0_{2g\times k-g} & C_{2g\times 2g} \ep,\]
where $B$ is a  $(k-g)\times (k-g)$ matrix  and $C$ is a $2g\times 2g$-matrix.
It now follows that $A_K(t)$ and the $2g\times 2g$--matrix $C$ represent the same element in $W(\zt\to \Q(t))$.
\end{proof}

The classical invariant $\frac12 \rho_{\zt}(K)$ gives a lower bound on the topological 4--ball genus and thus on the algebraic unknotting number.
The following lemma now says, that as a lower bound on $u_a(K)$, the invariant $\frac12 \rho_{\zt}(K)$ is subsumed by $n(K)$.

\begin{lemma}
For any knot $K$ we have
\[ n(K)\geq \frac{1}{2}\rho_{\zt}(K).\]
\end{lemma}

\begin{proof}
Recall that we   denote by $\MM$ the set of all hermitian matrices $A$ over $\zt$ such that $\det(A(1))=\pm 1$ and  we write  $A\sim B$ if $\l(A)$ and $\l(B)$ are isometric.
We now consider the map
\[ \ba{rcl} \MM&\to & W(\zt\to \Q(t)) \\
A(t)&\mapsto & A(t)\oplus -A(1).\ea \]
Note that it is well--known that given a hermitian matrix $D(t)$ over $\zt$ with $\det(D(t))=\pm 1$
the pairings $D(t)$ and $D(1)$ define the same element in $W(\zt\to \Q(t))$ (see \cite{Ran81} for details).
It now follows from Proposition \ref{prop:ra81} that the above map
descends to a map
\[ \MM/\sim \,\,\to \,W(\zt \to \Q(t)).\]
The lemma now follows immediately from the definitions.
\end{proof}

\begin{remark}
In \cite{Liv11} Livingston shows that $\rho(K)$ is completely determined by the Levine--Tristram signatures.
In an interesting twist  Livingston \cite[Section~3.1]{Liv11}  gives an example which shows that in general
$\rho(K)\geq \mu(K)$.
The invariant $\mu(K)$ is thus \emph{not} the optimal lower bound on the algebraic unknotting number which can be obtained from the Levine--Tristram signatures and the nullities.
\end{remark}

\subsection{The unknotting number one obstruction by Fogel--Murakami--Rickard}\label{section:fogelmurakami}

The following unknotting number one obstruction was proved by
H. Murakami \cite{Muk90} and Fogel \cite[p.~32]{Fo93} and it was already known to John Rickard \cite{Lic11}.

\begin{theorem}\label{thm1u1}\label{thm:u1}
Let $K$ be a knot and let $\eps \in \{-1,1\}$. If  $K$ can be turned into an Alexander polynomial one knot using one $\eps$--crossing change,
then  there exists a generator $g$ of $H_1(X(K);\zt)$
such that
\[ \l(g,g)=\frac{-\eps}{\Delta_K(t)}\in \Q(t)/\zt.\]
\end{theorem}

We will now see that it is an almost immediate corollary to Theorem \ref{thm1}.

\begin{proof}
It follows from Theorem \ref{thm1} that $\l(K)\cong \l(p(t))$ for a polynomial $p(t)$ with $p(1)=-\eps$.
Since $p(t)$ represents the Alexander module and since $\Delta_K(1)=1$ it follows that $p(t)=-\eps\Delta_K(t)$.
In particular there exists a generator $g$
 of $H_1(X(K);\zt)$
such that
\[ \l(g,g)=\frac{1}{-\eps\Delta_K(t)}=\frac{-\eps}{\Delta_K(t)}\in \Q(t)/\zt.\]
\end{proof}

\subsection{The unknotting number one obstruction by Lickorish} \label{section:lickorish}

The following theorem
was proved by Lickorish \cite{Lic85} (see also \cite[Proposition~2.1]{CoL86}).

\begin{theorem}\label{thm:lickorish}
Let $K$ be a knot and let $\eps\in \{-1,1\}$. If $K$ can be unknotted using one $\eps$--crossing change,
then  there exists a generator $h$ of $ H_1(\Sigma(K);\Z)$ such that
\[ l(h,h)=\frac{-2\eps }{\det(K)}\in \Q/\Z.\]
\end{theorem}

We will now show that if a knot satisfies the conclusion of Theorem \ref{thm:u1},
then the Lickorish obstruction vanishes.
This shows in particular that the Lickorish obstruction gives in fact an obstruction to the \emph{algebraic} unknotting number being equal to one.

\begin{theorem}\label{thm:n1lick}\label{thm:comparison2}
Let $K$ be a knot and let $\eta \in \{-1,1\}$.
Suppose there exists
 a generator $k$ of $H_1(X(K);\zt)$
such that
\[ \l(k,k)=\frac{\eta}{\Delta_K(t)}\in \Q(t)/\zt.\]
 Then  there exists a generator $h$ of $ H_1(\Sigma(K);\Z)$ such that $l(h,h)=\frac{2\eta }{\det(K)}\in \Q/\Z$.
\end{theorem}

\begin{proof}
Suppose there exists
 a generator $k$ of $H_1(X(K);\zt)$
such that
\[ \l(k,k)=\frac{\eta}{\Delta_K(t)}=\frac{1}{\eta \Delta_K(t)}\in \Q(t)/\zt\]
for some $\eta\in \{-1,1\}$.
This is equivalent to saying that $\l$ is isometric to $\l(\eta\Delta_K(t))$.
 It follows from Lemma \ref{lem:atlinkingform}  that
\[ 2l(K)\cong  l(\eta\Delta(-1))=l(\eta \det(K))\in \Q/\Z.\]
This means that there exists a generator $g$ for $H_1(\Sigma(K);\Z)$
such that
\[ 2l(g,g)=\frac{\eta}{\det(K)}=\frac{1}{\eta \det(K)}\in \Q/\Z.\] Since $\det(K)$ is an odd number
it follows that $k=2g$ is also a generator for $H_1(\Sigma(K);\Z)$, and it is easy to see that it has the required properties.
\end{proof}

\begin{remarks}
\bn
\item We could also easily have deduced Theorem \ref{thm:n1lick} from Theorem \ref{thm:representlk}. Put differently,
the Lickorish obstruction is precisely the obstruction of Theorem \ref{thm:representlk} to $n(K)=1$.
\item
Stoimenow \cite[p.~763~and~Conjecture~7.4]{St04} conjectures that the Lickorish obstruction contains the obstructions to
$u(K)=1$ obtained from the Jones polynomial which was found by
Miyazawa \cite{Mi98}, Traczyk \cite{Tra99}  and Stoimenow \cite{St04}.
\en
\end{remarks}

\subsection{The unknotting number one obstructions by Jabuka} \label{section:jabuka}

In the following we denote by $W(\Q)$ the Witt group of non-singular bilinear symmetric pairings over $\Q$.
Note that we can think of $W(\Q)$ also as the Witt group of symmetric matrices over $\Q$ with non-zero determinant. We refer to \cite[Section~2]{Ja09}, \cite{HM73} and \cite{Ran81} for details.
Given a knot $K$  Jabuka \cite{Ja09} denotes
by $\varphi(K)$ the element in the Witt group $W(\Q)$ defined by $V_K+\vt_K$.
The following is now Jabuka's obstruction to $u(K)=1$:

\begin{theorem}\label{thm:jabuka}
Let $K$ be a knot and let $\eps \in \{-1,1\}$. If $K$ can be unknotted using one $\eps$--crossing change,
then
$\varphi(K)$ is represented by the diagonal matrix with entries $2\eps $ and $-2\eps\det(K)$.
\end{theorem}

\begin{remark}
The statement of Theorem \ref{thm:jabuka} is precisely the statement of \cite[Corollary~1.2]{Ja09},
the only difference is that we view the determinant of a knot as a signed invariant,
i.e. we write $\det(K)=\det(V+V^t)$, whereas Jabuka uses $|\det(V+V^+)|$ as the definition of the determinant of a knot.
 \end{remark}

Note that if a knot $K$ can be turned into an Alexander polynomial one knot using one $\eps$--crossing change,
then it follows from Theorem \ref{thm:atsignatures} that $\s(K)\in \{0,2\eps\}$.
The following result thus shows   that the Lickorish obstruction together with the signature obstruction  subsumes the Jabuka obstruction.

\begin{theorem}\label{thm:lickorishjabuka}
Let $K$ be a knot and let $\eps\in \{-1,1\}$.
If $\s(K)\in \{0,2\eps\}$ and if  there exists a generator $h$ of $ H_1(\Sigma(K);\Z)$ such that $l(h,h)=\frac{-2\eps }{\det(K)}\in \Q/\Z$,
 then the conclusion of Theorem \ref{thm:jabuka} also holds.
\end{theorem}

\begin{remark}
  Jabuka \cite{Ja09} showed that in general the Lickorish obstruction is stronger than the obstruction provided by Theorem \ref{thm:jabuka}, e.g. the Jabuka obstruction vanishes for $K=8_8$, but the Lickorish obstruction detects that $u(8_8)\geq 2$.
\end{remark}

In our proof of Theorem \ref{thm:lickorishjabuka} we will need  the following well--known elementary lemma:

\begin{lemma}\label{lem:detsign}
Let $K$ be a knot, then
\[ \sign(\Delta_K(-1))=(-1)^{\s(K)/2}.\]
\end{lemma}

We provide a proof for the reader's convenience.

\begin{proof}
Let $V$ be a Seifert matrix for $K$. Without loss of generality we can assume that $V$ is a $4k\times 4k$--matrix.
We denote by $p$ the number of positive eigenvalues of $V+\vt$ and we denote by $n$ the number of negative eigenvalues of $V+\vt$. It follows that
\[ \ba{rcl} \sign(\Delta_K(-1))&=&\sign((-1)^{-2k}\det(-V-\vt))\\
&=&\sign(\det(V+\vt))=(-1)^n=(-1)^{\frac{n-p}{2}+\frac{n+p}{2}}\\
&=&(-1)^{-\s(K)/2}\cdot (-1)^{2k}\\
&=&(-1)^{\s(K)/2}.\ea \]
\end{proof}

\begin{proof}[Proof of Theorem \ref{thm:lickorishjabuka}]
We will use the notation of the proof of Theorem \ref{thm:n1lick}.
We denote by $W(\Z)$ the Witt group of non-singular pairings over $\Z$.
Note that the signature defines an isomorphism
\be \label{equ:signiso} \sign\colon W(\Z)\to \Z \ee
We refer to \cite{HM73} for details.
We say that a linking pairing $H\times H\to \Q/\Z$  is  metabolic if there exists a subspace $P\subset H$ with $P=P^\perp$.
We  denote by $W(\Z,\Q)$ the Witt group of linking pairings modulo metabolic pairings.

By definition  any pairing in $W(\Q)$ can be represented by a rational matrix. After multiplying by a sufficiently
large square we can also represent a given pairing by an integral symmetric $n\times n$--matrix $A$. To such a matrix we then associate the linking pairing $l(A)$.
Note that $l(-A)=-l(A)$ in the group $W(\Z,\Q)$. Also note that the above assignment descends to a well-defined map $W(\Q)\to W(\Z,\Q)$ and it is well-known that the sequence
\be \label{equ:seswitt} 0\to W(\Z)\to W(\Q)\to W(\Z,\Q)\to 0\ee
is exact. We refer to \cite[p.~90]{HM73} and to \cite{Ran81} for details.

Now let $K$ be a knot and $V$ a Seifert matrix for $K$. Recall (see e.g. \cite{Go78})
 that the linking pairing $l=l(K)$ is isometric to  $l( V+\vt)$. Now suppose we have $\eps\in \{-1,1\}$
 such that the following hold:
 \bn
 \item  $\s(K)\in \{0,2\eps\}$,
  \item  there exists a generator $h$ of $ H_1(\Sigma(K);\Z)$ such that $l(h,h)=\frac{-2\eps }{d}\in \Q/\Z$, where we write $d=\det(K)$.
 \en
We can now prove the following claim.

\begin{claim}
The element in $W(\Q)$ represented by
$(V+\vt)\oplus (2\eps/d)\oplus (-2\eps)$
gets sent to the trivial element in $W(\Z,\Q)$.
\end{claim}

In the following we identify the linking pairing $l(K)$ with the pairing
on $\Z/d$ given by $l(a,b)=\frac{-2\eps ab}{d}\in \Q/\Z$.
Now recall that the image of $(V+\vt)\oplus (2\eps/d)\oplus (-2\eps)$
is represented by the pairing $l(K)\oplus l(2\eps d)\oplus l( -2\eps)$.
We consider the map
\[ \ba{rcl} \Z/2 \oplus \Z/d &\to & \Z/2d \\
(x,y)&\mapsto & xd+2y.\ea \]
It is straightforward to verify that this map induces an isometry
\[ l(- 2\eps )\oplus l(K) \to l(-2\eps d).\]
Put differently, the pairing $l(K)\oplus l( 2\eps/ d)\oplus l(-2\eps)$
represents the trivial element in $W(\Z,\Q)$. This concludes the proof of the claim.

It follows from Lemma \ref{lem:detsign} and from $\s(K)\in \{0,2\eps\}$
that the signature of the matrix $(V+\vt)\oplus (2\eps/d)\oplus (-2\eps)$ is zero.
It now follows from the claim, the short exact sequence (\ref{equ:seswitt}) and (\ref{equ:signiso}) that
$(V+\vt)\oplus (2\eps/d)\oplus (-2\eps)$ represents the zero element in $W(\Q)$.
The fact that $(2\eps/d)=(2\eps\cdot d)\in W(\Q)$ now completes the proof of the theorem.
\end{proof}

\subsection{The unknotting number two obstruction by Stoimenow} \label{section:stoimenow}

Stoimenow proved the following theorem.

\begin{theorem}[Stoimenow, \expandafter{\cite[Theorem~5.2]{St04}}]
Let $K$ be a knot with  $|\s(K)|=4$  such that $\det(K)$ is a square.
If $\det(K)$ has no divisors of the form $4r+3$, then $u(K)>2$.
\end{theorem}

Our next theorem shows that the $n(K)$ obstruction contains the Stoimenow obstruction.
This shows in particular that the Stoimenow obstruction is an obstruction to the algebraic unknotting number being equal to two.
The latter result can also be shown by reading carefully the original proof.

\begin{theorem}\label{thm:stoimenow}
Let $K$ be a knot  with  $|\s(K)|=4$ such that $\det(K)$ is a square.
If $n(K)=2$, then $\det(K)$ has a divisor of the form $4k+3$.
\end{theorem}

\begin{proof}
Let $K$ be a knot  with  $|\s(K)|=4$ and such that $\det(K)$ is a square and suppose that $n(K)=2$.
Without loss of generality we can assume that $\s(K)=4$.
By Theorem  \ref{thm:representlk} there exists a positive definite matrix $A$ with $|\det(A)|=|\det(K)|$
such that
\[ A=\bp 4k+3  &2m \\ 2m & 4l+3 \ep,\]
for some $k,l,m\in \Z$.
Note that $A$ being positive definite and $\det(K)$ being a square implies that in fact $\det(A)=\det(K)$.

Since $A$ is positive definite it also follows that $4k+3>0$ and that
\[ \det(A)=(4k+3)(4l+3)-4m^2 \]
is positive. It follows that $k$ and $l$ actually are positive.
Now assume that $\det(A)=\det(K)=d^2$ is a square.
We thus see that $(4k+3)(4l+3)-4m^2=d^2$. Since $4k+3>0$ we can find a prime $p$  of the form $4r+3$ which divides $4k+3$.
We are done once we show that $p$ also divides $m$.

Suppose that $p$ does not divide $m$. Since $p$ divides $d^2+4m^2$ we obtain the equation
$d^2=-(2m)^2 \mod p$, but since $2m$ is non--zero, and hence invertible modulo $p$ we obtain that $-1$ is a square modulo $p$.
But it is well--known that $-1$ is not a square modulo a prime of the form $4r+3$.

\end{proof}


\section{New obstructions from the Blanchfield pairing}

\subsection{Obstructions from the Blanchfield pairing to $n(K)=1$}

We have already seen that the Nakanishi index, the signature and the Lickorish criterion give lower bounds
on $n(K)$. The Lickorish obstruction can be summarized as replacing an infinite problem
(can the Blanchfield pairing be represented by a $1\times 1$--matrix over $\zt$) by a finite problem (can the linking pairing be represented by a $1\times 1$--matrix over $\Z$).
This principle can easily be generalized.

To formulate the generalizations we need two `reductions' of the Blanchfield pairing.
First, let $p$ be a prime. We denote by $Q_p$
the quotient field of $\F_p\tpm$. Then we can imitate the definition of the Blanchfield pairing over $\zt$ to define a  pairing
\[ H_1(X(K);\F_p\tpm)\times H_1(X(K);\F_p\tpm)\to Q_p/\F_p\tpm.\]
Second, let  $k$ be an integer such that $H_1(\S_k(K))$ is finite, where $\S_k(K)$ denotes the $k$--fold branched cover of $K$. Note that $H_1(\S_k(K))$ is a module over $\Z[\Z/k]$, the group ring of $\Z/k$.
We can then define a pairing
\[ H_1(\S_k(K)) \times H_1(\S_k(K))\to S^{-1}\Z[\Z/k]\,/\,\Z[\Z/k],\]
where
\[ S:=\{ f\in \Z[\Z/k]=\zt/(t^k-1)\,|\, f(1)=1\}.\]
The proof that the matrix $A_K(t)$ over $\zt$ (see equation \eqref{equ:akt}) is a presentation matrix
of the Blanchfield pairing over $\zt$ can also be modified easily to show that
 $A_K(t)$ viewed as a matrix over $\F_p\tpm$ respectively over $\Z[\Z/k]=\zt/(t^k-1)$ is a presentation matrix for the two above pairings.
 (In particular this shows that both pairings are classical invariants, see also \cite{Go78} for more information.)
We thus obtain the following lemma:

\begin{lemma}\label{lem:finitelinking}
Let $K$ be a knot with $n(K)=n$.
\bn
\item Let $p$ be a prime, then the $\F_p$--Blanchfield pairing
\[ H_1(X(K);\F_p\tpm)\times H_1(X(K);\F_p\tpm)\to Q_p/\F_p\tpm)\]
can be represented by an $n\times n$--matrix over $\F_p\tpm$.
\item  Let $k$ be an integer, then the $\Z[\Z/k]$--Blanchfield pairing
\[ H_1(\S_k(K)) \times H_1(\S_k(K))\to S^{-1}\Z[\Z/k]\,\,/\,\,\Z[\Z/k]\]
can be represented by an $n\times n$--matrix over $\Z[\Z/k]$.
\en
\end{lemma}

If $p$ is any prime, or if $k$ is any integer such that $H_1(\S_k(K);\Z)$ is finite,
then we are dealing with finite objects. In particular in theses cases the obstructions provided by the lemma are computable.
We implemented both obstructions in the case $n(K)=1$.
We applied it to all knots with up to 14 crossings with primes usually up to 11 and $k$ usually up to 6.
(The size of $H_1(\S_k(K))$ typically grows very fast with $k$, putting limitations on the range of $k$.)
 To our great surprise (and disappointment), among all knots with up to 14 crossings  we could not find a single example
where these obstructions to $n(K)=1$ could see beyond the Nakanishi index, the Levine--Tristram signatures and the Lickorish obstruction.

In  Sections~\ref{section:obstwo} and \ref{section:obsthree}
we will on the other hand see that using the linking pairing on $H_1(\Sigma(K))$ we can give new obstructions to $n(K)=2$ and $n(K)=3$,
which actually do work in practice.

\subsection{Obstructions from the linking pairing to $n(K)=2$}\label{section:obstwo}

%

If $K$ is a knot with $n(K)=2$, then it follows from Theorem \ref{thm:representlk} that the linking pairing $l(K)$
can be represented by a certain symmetric $2\times 2$--matrix.
The full classification of all symmetric $2\times 2$ matrices up to a congruence was already known to Gau\ss,
below we state a slightly weaker result, referring to \cite[Section 15.3]{CS99} for an excellent exposition.

\begin{lemma}\label{lem:22indef}
Let $A$ be a symmetric integral $2\times 2$--matrix with $d:=\det(A)\ne 0$. Then, either $A$ is congruent to a matrix of the form
\[ \bp a & c \\ c& b\ep \]
such that the following hold:
\bn
\item $0<|a|\leq |b|\leq |d|$,
\item $c\in \{0,\dots,\lfloor \frac{|a|}{2}\rfloor\}$,
\en
or $A$ is congruent to a matrix of the form
\[ \bp a & c \\ c& 0\ep \]
with $c^2=d$, $c\geq 0$ and $a\in \{-c,\dots,c\}$.
\end{lemma}

For the reader's convenience
 we give a short proof of the lemma.

\begin{proof}
First assume that  $A$ is congruent to a matrix such that one of the diagonal entries is zero.
It is straightforward to see that in that case $A$ is congruent to a matrix of the latter type.

Now suppose that $A$ is not congruent to matrix such that one of the diagonal entries is zero.
Among all matrices congruent to $A$ we then pick a matrix such that the absolute value of the (1,1) entry (i.e. the top left one) is minimal.
We write this matrix as
\[ B:=\bp a & c \\ c& b\ep. \]
After adding a suitable multiple of the first row to the second row and the same multiple of the first column to the second column
we can assume that $|c|\leq \frac{|a|}{2}$. If $c<0$, then we multiply the first row and the first column by minus one, to arrange
that $c\geq 0$. By the minimality of $a$, even after these operations, we still have that $|a|\leq |b|$.

Finally note that
\[ \ba{rcl} |d|=|ab-c^2|&\geq &|a||b|-(\lfloor \frac{|a|}{2}\rfloor)^2 \\
&\geq &|a||b|-(|a|-1)^2\\
&=&|b|+(|a|-1)|b|-(|a|-1)^2=|b|+(|a|-1)(|b|-|a|+1)\geq|b|.\ea\]
We thus see that $|d|\geq |b|$.
\end{proof}

We can now describe our obstruction to a knot having $n(K)=2$.
Let $K$ be a knot with determinant $d=\det(K)$.
We denote by  $C_1,\dots,C_l$ the matrices which satisfy conditions $(1)$ and $(2)$ from the previous corollary
applied to $\pm d$,
and which are congruent to the identity modulo two.
If $\s(K)=4\cdot \eps$ for some $\eps\in \{-1,1\}$, then we furthermore demand that each $C_i$ is $\eps$--definite
 and that each $C_i$ is congruent modulo four to a matrix of the form
\[ \bp -\eps & 2m \\ 2m&-\eps\ep , \]
for some $m\in \Z$. It is clear that this list of matrices $C_1,\dots,C_l$ can be explicitly determined.

We can now formulate the following obstruction:

\begin{proposition}
If $n(K)=2$, then there exists an integer $k\in \{1,\dots,l\}$ and an isometry $l(C_k)\cong 2l(K)$.
\end{proposition}

Note that it follows easily from the proof of Theorem \ref{thm:stoimenow} that  this $n(K)=2$ obstruction contains the Stoimenow obstruction.

\begin{proof}
It is obvious that congruent matrices define isometric linking pairings. The proposition now follows from Theorem  \ref{thm:representlk} and Lemma \ref{lem:22indef}.
\end{proof}

The following lemma gives an elementary way to check whether $2l(K)$ is isometric to $l(C)$ for a given $2\times 2$--matrix $C$:

\begin{lemma}\label{lem:checkn2}
Let $C$ be a symmetric integral $2\times 2$--matrix with $\det(C)=\pm \det(K)$.
Then there exists an isometry $l(C)\cong 2l(K)$ if and only if there exist
$v_1,v_2\in H_1(\Sigma(K))$ which generate $H_1(\Sigma(K))$, such that
\[ 2l(K)(v_i,v_j)=\mbox{$(i,j)$--entry of $C^{-1}$} \mbox{ modulo $\Z$}\]
for any $i,j\in \{1,2\}$.
\end{lemma}

\begin{proof}
First let $\Phi\colon \Z^2/C \Z^2\to H_1(\Sigma(K))$ be an isomorphism which induces an isometry
$l(C)\cong 2l(K)$.
We denote by $v_1,v_2$ the images of $e_1=(1,0)$ and $e_2=(0,1)$. It follows immediately from the definitions
that $v_1$ and $v_2$ have the desired properties.

Conversely, suppose we are given $v_1,v_2\in H_1(\Sigma(K))$ which generate $H_1(\Sigma(K))$, such that
\be \label{equ:formsij} 2l(K)(v_i,v_j)=\mbox{$(i,j)$--entry of $C^{-1}$} \mbox{ modulo $\Z$}\ee
for any $i,j$.
We denote by $\Phi\colon \Z^2\to H_1(\Sigma(K))$ the map given by $\Phi(e_i)=v_i$.
This map is evidently surjective.

We claim that this map descends to a map $\Z^2/C \Z^2\to H_1(\S(K))$.
Let $v\in C\Z^2$. Note that $v^tC^{-1}w\in \Z$ for all $w\in \Z^2$. It now follows that $l(K)(\Phi(v),\Phi(w))=0\in \Q/\Z$ for all $w\in \Z^2$.
But since $\Phi$ is surjective and since $l(K)$ is non--degenerate this implies that $\Phi(v)=0\in H_1(\Sigma(K))$. This shows
that $\Phi$ descends to a map $\Z^2/C \Z^2\to H_1(\Sigma(K))$.

Since the map $\Phi$ is an epimorphism between finite groups of the same size it follows that this map is an isomorphism,
and  condition (\ref{equ:formsij})  implies that $\Phi$  is in fact an isometry.
\end{proof}

In Section~\ref{ss:examplesnk2} we will give several examples of knots where this obstruction applies.


\subsection{Obstructions from the linking pairing to $n(K)=m$ for $m\ge 3$}\label{section:obsthree}

The approach in the previous section can also be extended to give an obstruction to $n(K)=m$ for arbitrary $m$. We focus on the definite case (i.e. when
the absolute value of the signature is equal to $2m$), and we plan to deal with the general case in a future paper.
The key ingredient to getting obstructions is  Lemma \ref{lem:33matrix} below.

\begin{lemma}\label{lem:33matrix}
Let $A$ be a positive definite symmetric $m\times m$--matrix. Then $A$ is congruent to a matrix of the form
\[ \bp f_{11} & f_{12} & \dots & f_{1m}\\ f_{21} & f_{22} & \dots & f_{2m}\\ \vdots & \vdots & \ddots & \vdots \\ f_{m1} & f_{m2} & \dots & f_{mm}\ep\]
such that
\bn
\item $0<f_{11}\le f_{22}\le \dots \le f_{mm}$,
\item for any $1\le i<j\le m$ we have $|2f_{ij}|\le f_{ii}$,
\item $f_{mm}\le B_m\det A$ for a constant $B_m$ depending only on $m$,
\item furthermore we can take $B_1=B_2=B_3=B_4=1$.
\en
\end{lemma}

\begin{proof}
By \cite[Theorem 12.1.1]{Ca78} the matrix $A$ can be put into a so-called reduced form (in the sense of Minkowski). Then \cite[Lemma 12.1.1]{Ca78}
shows parts (1) and (2) of Lemma~\ref{lem:33matrix}.

To show (3), observe that for a matrix in a reduced form we have
\be\label{eq:vdW}
f_{11}\cdot f_{22}\cdot \ldots \cdot f_{mm}\le \mu_m\det A,
\ee
where $\mu_2=\frac{4}{3}$, $\mu_3=2$, $\mu_4=4$ and for $m>4$,
\[\mu_m=\left(\frac{2}{\pi}\right)^m\left\{\Gamma\left(2+\frac{n}{2}\right)\right\}^{2}\left(\frac{5}{4}\right)^{\frac{1}{2}(n-3)(n-4)},\]
where $\Gamma$ is the Euler Gamma function
(see an excellent survey \cite{Wa56} for proofs and details.) As $f_{ii}\ge 1$, we immediately obtain that $f_{mm}\le B_m\det A$ for $B_m=\mu_m$.

To show that for $m=2,3,4$ we have $f_{mm}\le \det A$ we again use \eqref{eq:vdW}. Let us begin with the case $m=2$. If $f_{11}\ge 2$,
by \eqref{eq:vdW} we get $f_{22}\le \frac23\det A$. So assume that $f_{11}=1$. Then $f_{12}=f_{21}=0$ because the matrix is in the reduced form (see point
(2) in the statement of the lemma). But then $\det A=f_{22}$.

For $m=3$, if $f_{11}>1$ then $f_{11}f_{22}\ge 4$, so $f_{33}\le \frac12\det A$. If $f_{11}=1$, we have $f_{12}=f_{13}=f_{21}=f_{31}=0$, so the form is a block sum of
$(1)$ and a two dimensional form and we use the case $m=2$. The argument with $m=4$ is identical.
\end{proof}

If we assume that $f_{ii}\cong 3\bmod 4$ (cf. Theorem~\ref{thm:representlk}), we can show that $B_m$ can be chosen to be $1$ for some higher $m$ as well.
\begin{lemma}
If $f_{11},\dots,f_{mm}$ are congruent to $-1$ modulo $4$, then $f_{mm}\le \det A$ for $m\le 7$.
\end{lemma}
\begin{proof}
From the assumptions $f_{ii}\ge 3$, hence by \eqref{eq:vdW} we have $f_{mm}\le 3^{1-m}\mu_m\det A$. Now an explicit computation shows that $3^{1-m}\mu_m\le 1$
for $m=5,6,7$.
\end{proof}

We can now describe our obstruction to a knot with $|\s(K)|=2m$ having $n(K)=m$.
Let $K$ be a knot. We write $d=|\det(K)|$
 and $|\s(K)|=2m$. Without loss of generality we can assume
that $\s(K)=2m$.
We denote by  $C_1,\dots,C_r$ the positive definite $m\times m$--matrices with determinant $d$ which satisfy conditions $(1), (2)$, $(3)$ and $(4)$ from Lemma \ref{lem:33matrix},
and which are congruent to the identity modulo two.
We furthermore restrict ourselves to matrices which are congruent to $-1$ modulo $4$.  It is clear that this list of matrices $C_1,\dots,C_r$ can be explicitly determined, even though for large $m$ or $\det A$
this list may be very long.

We can now formulate the following obstruction,
which is an immediate consequence of Theorem \ref{thm:representlk} and Lemma \ref{lem:33matrix}.

\begin{proposition}
If $n(K)=m$, then there exists an integer $s\in \{1,\dots,r\}$ and an isometry $l(C_s)\to 2l(K)$.
\end{proposition}

The following lemma gives an elementary way to check whether $2l(K)$ is isometric to $l(C)$ for a given $m\times m$--matrix $C$:

\begin{lemma}
Let $C$ be a symmetric $m\times m$--matrix with $\det(C)=\pm \det(K)$.
Then there exists an isometry $l(C)\cong 2l(K)$ if and only if there exist
$v_1,v_2,\dots,v_m\in H_1(\Sigma(K))$ which generate $H_1(\Sigma(K))$, such that
\[ 2l(K)(v_i,v_j)=\mbox{$(i,j)$--entry of $C^{-1}$} \mbox{ modulo $\Z$}\]
for any $i,j\in \{1,\dots,m\}$.
\end{lemma}

The proof is of course almost identical to the proof of Lemma \ref{lem:checkn2}. Examples of application of this
criterion are given in Section~\ref{section:exampleu3}.

\subsection{Comparison with Owens' obstruction}
We now give a comparison of our obstruction with  the Owens obstruction \cite[Theorems~1~and~5]{Ow08}.
We summarize the key facts:
\bn
\item Owens shows that if a knot satisfies  $u(K)=m$ and  $|\s(K)|=2m$,
 then the Heegaard--Floer correction terms (see \cite{OS03a}) of the 2--fold branched cover of $K$
 satisfy
    \bn
    \item a certain inequality,
    \item and a certain equality modulo 2.
    \en
\item By work of Ozsv\'ath--Szab\'o \cite[Theorem~1.2]{OS03a} and  Taylor \cite{Ta84} (see also \cite{OwS05})  a knot which satisfies the  `mod 2 equality' of
Owens also  satisfies the conclusion of Theorem \ref{thm:representlk}.
\en
In practice the fact that one needs to be able to calculate the Heegaard--Floer correction terms of the 2--fold branched cover
means that the Owens obstruction can be calculated in a straightforward way for alternating knots, but it is rather difficult to calculate for most other knots. (Note though that calculations can be made if the 2--fold branched cover is a Seifert fibered space, this is the case for Montesinos knots and torus knots.)
We conclude with the discussion of some examples:
\bn
\item Owens \cite{Ow08} shows that the $u(K)=2$ obstruction applies to the alternating knots $9_{10}, 9_{13},  9_{38}, 10_{53}, 10_{101}, 10_{120}$ which have signature equal to four.
On the other hand the algebraic unknotting number equals 2 (using the computer program `knotorious' we can find explicit algebraic unknotting operations  in the sense of \cite{Muk90}).
\item Owens furthermore uses Heegaard--Floer homology and a result of Traczyk to show that the unknotting number of the alternating knot $9_{35}$ equals three, even though the signature equals two.
Again, the algebraic unknotting number equals 2.
\item Using the obstruction of Section \ref{section:obstwo} we can show  that the algebraic unknotting number of the non--alternating knot $11n_{148}$ equals three, even though the signature equals two.
Since the knot is non--alternating and since the signature does not equal four it seems difficult to use the Owens approach to show that the unknotting number equals three.
\item Owens \cite[Corollary~6]{Ow08} also used the $u(K)=3$ obstruction  to show that the unknotting number of the two-bridge knot $K=S(51, 35)$ equals four.
 This knot passes our $u_a(K)=3$ obstruction, and in fact we can show that the algebraic unknotting number of $K$ equals three.
 \en

\section{Examples}\label{section:ex}

 \subsection{Knotorious}
In the following we write
\[ \ZZ:=\{z\in S^1\,|\, z^{1296}=1\}.\]
We have written a computer program `knotorious' (which is available from the authors' webpages, see \cite{BF12a})
which given a Seifert matrix calculates the following invariants:
\bn
\item the signature,
\item the Levine--Tristram signatures $\s_z(K)$ with $z\in \ZZ$,
\item the lower bounds on the Nakanishi index coming from $H_1(\S(K))$ and the Alexander module over the finite fields $\F_2,\F_3,\F_5$ and $\F_7$,
\item the Lickorish obstruction,
\item the Stoimenow obstruction,
\item the $u(K)=2$ obstruction coming from the linking pairing on $H_1(\S(K))$ (see Section \ref{section:obstwo}),
\item the $u(K)=3$ obstruction coming from the linking pairing on $H_1(\S(K))$ (see Section \ref{section:obsthree}).
\en
The program also attempts to compute, in a non--rigorous way, the invariant $\mu(K)$.
Furthermore, the program also finds
upper bounds on the algebraic unknotting number by finding  explicit
algebraic unknotting moves (we refer to \cite{Muk90} and \cite[Section~2]{Sae99} for details on algebraic unknotting
 moves). We  calculated the above invariants  for all knots with up to 12 crossings, the details can be found on  the authors' webpages \cite{BF12a}.
All the examples in the subsequent sections are based on the calculations with `knotorious'.

\subsection{Examples for the new $n(K)=2$ obstruction coming from $l(K)$}\label{ss:examplesnk2}
Let $K$ be a knot.
In  Section~\ref{section:obstwo} we showed that the linking from $l(K)$ on the homology of the 2--fold branched cover of $K$
 gives an obstruction to $u(K)\leq 2$. We applied this obstruction to all knots with up to 12 crossings.
We found that the following knots with $|\s(K)|=4$ (in fact with $\mu(K)=4$) and $m(K)\leq 2$ have $n(K)>2$:
\[ \ba{lllllll} 9_{49},& 11a_{123},& 11n_{133},&12a_{311},&12a_{386},& 12a_{433}, &12a_{561},\\
             12a_{563},&12a_{569},& 12a_{664},&12a_{683},& 12a_{725}, &12a_{780},&12a_{907}\\
             12n_{276},&12n_{494},& 12n_{496},&12n_{626}, &12n_{654}.\ea \]
Furthermore the following knots have $|\s(K)|\leq 2$ (in fact
 $\mu(K)\leq 2$) and $m(K)\leq 2$ but $n(K)>2$:
\[ \ba{lllllll} 10_{103},& 11n_{148}, & 12a_{327},&12a_{921},&12a_{1194}, &12n_{147}.\ea \]
We now discuss to what degree previous invariants detect the unknotting numbers of the above examples:
\bn
\item  The Stoimenow obstruction applies to $9_{49}, 11n_{133}, 12a_{664}$ and $12n_{276}$, but does not apply to any of the other knots.
To the best of our knowledge none of the other previous classical invariants  detect that $n(K)>1$.
\item Stoimenow \cite{St04} also used the Brandt-Lickorish-Millett-Ho polynomial (see \cite{Ho85,BLM86}) to give an obstruction to $u(K)=2$.
Stoimenow shows that this criterion implies  that $u(10_{103})>2$. We did not check this criterion for the above 12 crossing knots.
\item Note that most of the above knots (including $10_{103}$ but not $9_{49}$) are alternating, in that case the Rasmussen $s$--invariant and the Ozsv\'ath--Szab\'o $\tau$--invariant
agree (up to a scaling factor) with the signature, in particular they do not determine the unknotting number of the above 12 alternating knots.
\item The $s$--invariant has been computed  for all knots with up to 12 crossings (see \cite{CL11}), it detects the unknotting number for only one of the above non--alternating knots, namely $12_{n276}$.
\item The $\tau$--invariant has been calculated for all knots with up to 11 crossings by Baldwin and Gillam (see \cite{BG06}),
 it does not detect the unknotting number for any of the above non--alternating knots with up to 11 crossings.
 \item Arguably $11n_{148}$ is the most interesting example. Many invariants for non--alternating knots are very difficult to calculate (e.g. the Heegaard--Floer correction terms of the 2--fold branched cover, as in \cite{OS05} and \cite{Ow08}).
 The aforementioned calculations show that the $\tau$--invariant and the $s$--invariant do not detect the unknotting number of $11n_{148}$. Furthermore, several obstructions to $u(K)=2$
 (e.g. \cite{Ow08} and \cite{St04}) can be applied only if $|\s(K)|=4$, whereas we have $\s(11n_{148})=2$.
\en
 The webpage \emph{knotinfo} \cite{CL11} maintained by Cha and Livingston collects the unknotting information on knots up to eleven crossings.
According to this information  it was not known before these calculations that $u(11a_{123})=3$
and $u(11n_{148})=3$.

\subsection{Examples for the $n(K)=3$ obstruction coming from $l(K)$}\label{section:exampleu3}

Let $K$ be a knot.
In  Section~\ref{section:obsthree} we showed that the linking from $l(K)$
 gives  rise to an obstruction to $u(K)\leq 3$. We applied this obstruction to  all knots with up to 14 crossings with $|\s(K)|=6$
and found that it applies to precisely two knots, namely $14n_{12777}$
 and $14a_{4637}$.

 We expect that given any $m\in \N$ there exist examples of knots such that the obstruction introduced in Section~\ref{section:obsthree}
 detects that $u_a(K)=m$, but such that all previous classical lower bounds do not detect that $u_a(K)=m$.
 Finding such examples will obviously require a different method than the `brute force' search we have done with `knotorious'.

\subsection{Knots with up to 11 crossings}\label{section:ua10}

Our calculations with `knotorious' show that the Nakanishi index, the signature,
the Lickorish obstruction and the $n(K)=2$ obstruction of Section \ref{section:obstwo} completely determine the algebraic unknotting number for all knots with up to 11 crossings.

The full details are available on \cite{BF12a}. For the reader's convenience we provide in Table~\ref{tab:ua10} the algebraic unknotting number for all knots with up to 10 crossings.
The subscripts denote the way of obtaining the result:
\begin{itemize}
\item $1_u$ means that there exists a single algebraic unknotting move which
changes the knot into a knot with trivial  Alexander polynomial.
\item $2_L$ means that we use the Lickorish obstruction.
\item $k_\sigma$ for $k=2,3,4$ means that the signature detects the algebraic unknotting number and it is greater than one.
\item $2_w$ stands for the Wendt criterion, in particular the Nakanishi
index is equal to two.
\item $2_A$ means that the minimal number of generators of the Alexander module over $\F_2[t^{\pm}]$ is two.
\item $3_{S}$ means that  the Stoimenow obstruction  (see Section~\ref{section:stoimenow}) applies.
\item $\mathbf{3_n}$ denotes our new obstruction as discussed in Section~\ref{section:obstwo}.
\end{itemize}
(Note that in some cases two or more obstructions will detect that $u_a(K)$, but we will only indicate one obstruction.)

The star indicates that the unknotting number is actually known (according to \cite{CL11}) to be
 larger than $u_a(K)$. We remark that for knots $10_{61}$, $10_{76}$ and $10_{100}$, we computed that $u_a(K)=2$, but it is not known, whether $u(K)=2$ or $3$.
These knots are not marked by a star.

\newcolumntype{K}{>{\columncolor[gray]{0.85}}r}
\begin{table}
\caption{Algebraic unknotting number for knots with up to 10 crossings. See Section~\ref{section:ua10} for an explanation of symbols.}\label{tab:ua10}
\begin{tabular}{||r|r||K|K||r|r||K|K||r|r||K|K||r|r||}
\hline\hline
knot&$n$&knot&$n$&knot&$n$&knot&$n$&knot&$n$&knot&$n$&knot&$n$\\\hline
$3_1$&$1_u$&
$9_2$&$1_u$&
$9_{38}$&$2_{\sigma}^*$&
$10_{25}$&$2_{\sigma}$&
$10_{61}$&$2_\sigma$&
$10_{97}$&$2_L$&
$10_{133}$&$1_u$\\

$4_1$&$1_u$&
$9_3$&$3_{\sigma}$&
$9_{39}$&$1_u$&
$10_{26}$&$1_u$&
$10_{62}$&$2_\sigma$&
$10_{98}$&$2_w$&
$10_{134}$&$3_\sigma$\\

$5_1$&$2_{\sigma}$&
$9_4$&$2_{\sigma}$&
$9_{40}$&$2_w$&
$10_{27}$&$1_u$&
$10_{63}$&$2_\sigma$&
$10_{99}$&$2_w$&
$10_{135}$&$1_u^*$\\

$5_2$&$1_u$&
$9_5$&$1_u^*$&
$9_{41}$&$2_w$&
$10_{28}$&$1_u^*$&
$10_{64}$&$1_u^*$&
$10_{100}$&$2_\sigma$&
$10_{136}$&$1_u$\\\hline

$6_1$&$1_u$&
$9_6$&$3_{\sigma}$&
$9_{42}$&$1_u$&
$10_{29}$&$2_L$&
$10_{65}$&$2_L$&
$10_{101}$&$2_\sigma^*$&
$10_{137}$&$1_u$\\

$6_2$&$1_u$&
$9_7$&$2_{\sigma}$&
$9_{43}$&$2_{\sigma}$&
$10_{30}$&$1_u$&
$10_{66}$&$3_{\sigma}$&
$10_{102}$&$1_u$&
$10_{138}$&$1_u^*$\\

$6_3$&$1_u$&
$9_8$&$1_u^*$&
$9_{44}$&$1_u$&
$10_{31}$&$1_u$&
$10_{67}$&$2_L$&
$10_{103}$&$\mathbf{3_{n}}$&
$10_{139}$&$3_\sigma^*$\\

$7_1$&$3_{\sigma}$&
$9_9$&$3_{\sigma}$&
$9_{45}$&$1_u$&
$10_{32}$&$1_u$&
$10_{68}$&$1_u^*$&
$10_{104}$&$1_u$&
$10_{140}$&$2_A$\\\hline

$7_2$&$1_u$&
$9_{10}$&$2_{\sigma}^*$&
$9_{46}$&$2_w$&
$10_{33}$&$1_u$&
$10_{69}$&$2_L$&
$10_{105}$&$2_L$&
$10_{141}$&$1_u$\\

$7_3$&$2_{\sigma}$&
$9_{11}$&$2_{\sigma}$&
$9_{47}$&$2_w$&
$10_{34}$&$1_u^*$&
$10_{70}$&$1_u^*$&
$10_{106}$&$2_L$&
$10_{142}$&$3_\sigma$\\

$7_4$&$2_L$&
$9_{12}$&$1_u$&
$9_{48}$&$2_w$&
$10_{35}$&$1_u^*$&
$10_{71}$&$1_u$&
$10_{107}$&$1_u$&
$10_{143}$&$1_u$\\

$7_5$&$2_{\sigma}$&
$9_{13}$&$2_{\sigma}^*$&
$9_{49}$&$3_S$&
$10_{36}$&$2_L$&
$10_{72}$&$2_\sigma$&
$10_{108}$&$2_L$&
$10_{144}$&$2_L$\\\hline

$7_6$&$1_u$&
$9_{14}$&$1_u$&
$10_{1}$&$1_u$&
$10_{37}$&$1_u^*$&
$10_{73}$&$1_u$&
$10_{109}$&$2_L$&
$10_{145}$&$1_u^*$\\

$7_7$&$1_u$&
$9_{15}$&$2_L$&
$10_2$&$3_\sigma$&
$10_{38}$&$1_u^*$&
$10_{74}$&$2_w$&
$10_{110}$&$1_u^*$&
$10_{146}$&$1_u$\\

$8_1$&$1_u$&
$9_{16}$&$3_{\sigma}$&
$10_3$&$2_L$&
$10_{39}$&$2_{\sigma}$&
$10_{75}$&$2_w$&
$10_{111}$&$2_\sigma$&
$10_{147}$&$1_u$\\

$8_2$&$2_{\sigma}$&
$9_{17}$&$2_L$&
$10_4$&$1_u^*$&
$10_{40}$&$2_L$&
$10_{76}$&$2_\sigma$&
$10_{112}$&$1_u^*$&
$10_{148}$&$1_u^*$\\\hline

$8_3$&$1_u^*$&
$9_{18}$&$2_{\sigma}$&
$10_5$&$2_\sigma$&
$10_{41}$&$1_u^*$&
$10_{77}$&$1_u^*$&
$10_{113}$&$1_u$&
$10_{149}$&$2_\sigma$\\

$8_4$&$1_u^*$&
$9_{19}$&$1_u$&
$10_6$&$2_\sigma^*$&
$10_{42}$&$1_u$&
$10_{78}$&$2_\sigma$&
$10_{114}$&$1_u$&
$10_{150}$&$2_\sigma$\\

$8_5$&$2_{\sigma}$&
$9_{20}$&$2_{\sigma}$&
$10_7$&$1_u$&
$10_{43}$&$1_u^*$&
$10_{79}$&$1_u^*$&
$10_{115}$&$2_A$&
$10_{151}$&$1_u^*$\\

$8_6$&$1_u^*$&
$9_{21}$&$1_u$&
$10_8$&$2_\sigma$&
$10_{44}$&$1_u$&
$10_{80}$&$3_\sigma$&
$10_{116}$&$2_L$&
$10_{152}$&$3_\sigma^*$\\\hline

$8_7$&$1_u$&
$9_{22}$&$1_u$&
$10_9$&$1_u$&
$10_{45}$&$1_u^*$&
$10_{81}$&$1_u^*$&
$10_{117}$&$1_u^*$&
$10_{153}$&$1_u^*$\\

$8_8$&$2_L$&
$9_{23}$&$2_{\sigma}$&
$10_{10}$&$1_u$&
$10_{46}$&$3_{\sigma}$&
$10_{82}$&$1_u$&
$10_{118}$&$1_u$&
$10_{154}$&$2_\sigma^*$\\

$8_9$&$1_u$&
$9_{24}$&$1_u$&
$10_{11}$&$1_u^*$&
$10_{47}$&$2_{\sigma}$&
$10_{83}$&$1_u^*$&
$10_{119}$&$1_u$&
$10_{155}$&$2_w$\\

$8_{10}$&$1_u^*$&
$9_{25}$&$1_u^*$&
$10_{12}$&$1_u^*$&
$10_{48}$&$1_u^*$&
$10_{84}$&$1_u$&
$10_{120}$&$2_\sigma^*$&
$10_{156}$&$1_u$\\\hline

$8_{11}$&$1_u$&
$9_{26}$&$1_u$&
$10_{13}$&$1_u^*$&
$10_{49}$&$3_{\sigma}$&
$10_{85}$&$2_\sigma$&
$10_{121}$&$2_L$&
$10_{157}$&$2_w$\\

$8_{12}$&$1_u^*$&
$9_{27}$&$1_u$&
$10_{14}$&$2_\sigma$&
$10_{50}$&$2_{\sigma}$&
$10_{86}$&$2_L$&
$10_{122}$&$2_L$&
$10_{158}$&$1_u^*$\\

$8_{13}$&$1_u$&
$9_{28}$&$1_u$&
$10_{15}$&$1_u^*$&
$10_{51}$&$1_u^*$&
$10_{87}$&$1_u^*$&
$10_{123}$&$2_w$&
$10_{159}$&$1_u$\\

$8_{14}$&$1_u$&
$9_{29}$&$1_u^*$&
$10_{16}$&$1_u^*$&
$10_{52}$&$1_u^*$&
$10_{88}$&$1_u$&
$10_{124}$&$4_\sigma$&
$10_{160}$&$2_\sigma$\\\hline

$8_{15}$&$2_{\sigma}$&
$9_{30}$&$1_u$&
$10_{17}$&$1_u$&
$10_{53}$&$2_{\sigma}^*$&
$10_{89}$&$2_L$&
$10_{125}$&$1_u^*$&
$10_{161}$&$2_\sigma^*$\\

$8_{16}$&$2_L$&
$9_{31}$&$2_L$&
$10_{18}$&$1_u$&
$10_{54}$&$1_u^*$&
$10_{90}$&$1_u^*$&
$10_{126}$&$1_u^*$&
$10_{162}$&$1_u^*$\\

$8_{17}$&$1_u$&
$9_{32}$&$1_u^*$&
$10_{19}$&$2_L$&
$10_{55}$&$2_\sigma$&
$10_{91}$&$1_u$&
$10_{127}$&$2_\sigma$&
$10_{163}$&$2_L$\\

$8_{18}$&$2_w$&
$9_{33}$&$1_u$&
$10_{20}$&$2_L$&
$10_{56}$&$2_\sigma$&
$10_{92}$&$2_\sigma$&
$10_{128}$&$3_\sigma$&
$10_{164}$&$1_u$\\\hline

$8_{19}$&$3_{\sigma}$&
$9_{34}$&$1_u$&
$10_{21}$&$2_{\sigma}$&
$10_{57}$&$1_u^*$&
$10_{93}$&$1_u^*$&
$10_{129}$&$1_u$&
$10_{165}$&$2_L$\\

$8_{20}$&$1_u$&
$9_{35}$&$2_w^*$&
$10_{22}$&$1_u^*$&
$10_{58}$&$1_u^*$&
$10_{94}$&$1_u^*$&
$10_{130}$&$1_u^*$&&\\

$8_{21}$&$1_u$&
$9_{36}$&$2_{\sigma}$&
$10_{23}$&$1_u$&
$10_{59}$&$1_u$&
$10_{95}$&$1_u$&
$10_{131}$&$1_u$&&\\

$9_1$&$4_{\sigma}$&
$9_{37}$&$2_w$&
$10_{24}$&$2_L$&
$10_{60}$&$1_u$&
$10_{96}$&$1_u^*$&
$10_{132}$&$1_u$&&\\
\hline\hline
\end{tabular}

\end{table}

\subsection{Knots with 12 crossings}\label{section:12}
Among the 12 crossing knots we found the following examples:
\bn
\item there exists precisely one knot, namely $12n_{749}$ with $|\s(K)|\leq 2$ and such that $|\s_z(K)|\geq 4$ for some $z\in \ZZ$,
\item there exists precisely one knot, namely $12a_{896}$ with $|\s_z(K)|\leq 2$ for all $z\in \ZZ$ and such that there exist $z_1,z_2\in \ZZ$ with
$|\s_{z_1}(K)-\s_{z_2}(K)|\geq 4$.
\en
Our calculations show that the aforementioned seven lower bounds determine the algebraic unknotting number for all knots with up to 12 crossings, except possibly for the following:
\[\begin{array}{cccccccc}
12a_{0050}&
12a_{0141}&
12a_{0364}&
12a_{0649}&
12a_{0728}&
12a_{0791}&
12a_{0901}\\
12a_{1049}&
12a_{1054}&
12a_{1064}&
12a_{1138}&
12a_{1141}&
12a_{1234}&
12a_{1236}\\
12a_{1264}&
12n_{0200}&
12n_{0260}&
12n_{0657}&
12n_{0864}.
\end{array} \]
The algebraic unknotting number of all the above  knots is either $1$ or $2$.

The knots $12a_{0050}$, $12a_{0141}$, $12a_{0364}$, $12a_{0649}$, $12a_{0728}$, $12a_{0791}$, $12a_{0901}$, $12a_{1054}$, $12a_{1064}$, $12a_{1138}$, $12a_{1234}$,
$12a_{1236}$, $12n_{0200}$ and $12n_{0864}$ all have Nakanishi index $1$. In fact, we were able to find an explicit generator of the Alexander module.
This allows us to compute the Blanchfield pairing for all those knots. Since $m(K)=1$ it is necessarily of the form
\[\xymatrix{
\L/p\L\times \L/p\L\ar[rrr]^{(v,w)\mapsto\ol{v}\cdot q\cdot w/p}&&& \O/\L,}
\]
where $p$ is the Alexander polynomial and $q\in\L$. For example, for $K=12a_{0050}$ we have
\[ \ba{rcl}
p&=&\Delta_K(t)=t^{-4}-8t^{-3}+20t^{-2}-30t^{-1}+33t-30t+20t^2-8t^3+t^4,\\
q&=&-t^3+7t^2-13t+17-13t^{-1}+7t^{-2}-t^{-3} \ea \]
(we refer to  \cite{BF12a} for the other knots).
Thus  $12a_{0050}$ has algebraic unknotting number $1$ if and only if there exists  an automorphism of $\L/p$ (as a $\L$-module),
which transforms this pairing into $(v,w)\mapsto \pm\ol{v}w/p$. This is equivalent to the existence of an  $f\in\L$ such that $qf\ol{f}=\pm 1\pmod p$. The problem of finding such $f$ or showing that it does not exist, in general, is very hard.

The knots $12a_{1054}$, $12a_{1141}$, $12a_{1264}$, $12n_{657}$ and $12n_{0260}$ have Nakanishi index 1 or 2. We were able to find a $2\times 2$ presentation
matrix in each case. For example, for $12a_{1054}$ we have
\[\begin{pmatrix} 2t^3& -1+4t-7t^2+4t^3-t^4\\ 1-5t+t^2-5t^3+t^4& 3t\end{pmatrix}.\]
However, we could not show that the Alexander module is cyclic. If it is not, then the algebraic unknotting number is $2$.

\section{Open questions}

Apart from Conjecture~\ref{conj:nua}, we can state a few more questions and problems related to $n(K)$.


\begin{question}
Given any knot $K$, do we have the following equality
\[ n(K)=\ba{c}\mbox{minimal size of a hermitian}\\
\mbox{matrix $A$ over $\zt$ with $\l(A)\cong \l(K)$}\ea\,\,\,?\]
Put differently, is the condition in the definition of $n(K)$ that $A(1)$ be diagonal over $\Z$ necessary?
Note that an affirmative answer would imply that $n(K)\le\deg\Delta_K(t)$ for any knot $K$.
\end{question}

\begin{question}
Let $K$ be a knot with $m(K)=1, |\s(K)|\leq 2$ and which satisfies the Lickorish obstruction.
Does it follow that  the $n(K)=1$ obstructions of Lemma \ref{lem:finitelinking} are necessarily satisfied? The discussion
following Lemma~\ref{lem:finitelinking} is evidence that the answer is yes.
\end{question}

\begin{question}
 Is $n(K)$ invariant under mutation? It is an open question whether the unknotting number is preserved under
mutation $($see \cite[Problem~1.69(c)]{Kir97}$)$.
The $S$--equivalence class of a Seifert matrix (and thus the isometry type of the Blanchfield pairing)
is  preserved under positive mutation $($see \cite[Theorem~2.1]{KL01}$)$. On the other hand the $S$--equivalence class
(in fact the isomorphism class of the Alexander module)
 is in general not preserved under negative mutation $($see \cite{Ke89} and \cite[Section~3]{Ke04}$)$.
We do not know whether $n(K)$ is preserved under mutation.
Note though that the Levine--Tristram signatures are preserved under any mutation $($see \cite{CL99}$)$
and note that homeomorphism type of the 2--fold branched cover is preserved under mutation (see e.g. \cite[Proposition~3.8.2]{Ka96}).
\end{question}

\begin{question}
In Section \ref{section:livingston} we introduce a new classical invariant $\rho_{\zt}(K)$
and we show in Proposition \ref{prop:rhozt} that it gives a lower bound on the topological 4--genus.
Can this invariant be used to give new computable lower bounds on the topological 4--genus?
\end{question}

\begin{question}
What are the algebraic unknotting numbers of the remaining 12 crossing knots? (See Section \ref{section:12} for details.)
It might be possible to possible to use the methods of \cite{KW03} to show that the Nakanishi index of the five 12 crossing knots mentioned
in Section \ref{section:12} is 2.
\end{question}

\end{document}